\documentclass[11pt]{amsart}
\usepackage{etex}

\usepackage{amsmath}

\usepackage{mathrsfs, euscript}
\usepackage{bbm}
\usepackage{amssymb}
\usepackage{tikz} 
\usepackage{amscd}

\usepackage[all,cmtip]{xy}

\usepackage[T1]{fontenc}  
\usepackage{textcomp}     

\usepackage{enumerate}

\newtheorem{theorem}{Theorem}[section]
\newtheorem*{theorem*}{Theorem}

\newtheorem{lemma}[theorem]{Lemma}
\newtheorem*{lemma*}{Lemma}
\newtheorem{corollary}[theorem]{Corollary}
\newtheorem{proposition}[theorem]{Proposition}
\newtheorem{defprop}[theorem]{Proposition-Definition}
\newtheorem{remark}[theorem]{Remark}
\newtheorem{definition}[theorem]{Definition}

\newcommand{\J}{\mathcal{J}}
\newcommand{\I}{\mathcal{I}}

\newcommand{\LL}{\mathcal{L}}

\newcommand{\f}{\mathcal{F}}
\newcommand{\bb}{\mathcal{B}}
\newcommand{\A}{\mathcal{A}}
\newcommand{\B}{B}
\newcommand{\E}{\mathbb{E}}

\newcommand{\C}{\mathbb{C}}

\newcommand{\mrx}{\mathring{\mathcal{X}}}
\newcommand{\X}{\mathcal{X}}

\newcommand{\talpha}{\tilde{\alpha}}

\usepackage[all]{xy}

\begin{document}
\unitlength=1mm
\special{em:linewidth 0.4pt}
\linethickness{0.4pt}
\title[Free-Boolean]{Free-Boolean independence with amalgamation}
\author{Weihua Liu}
\address{Weihua Liu: Department of Mathematics\\ Indiana University\\ Boomington, IN 47401, USA.\\ }
\email{liuweih@indiana.edu}
\author{Ping Zhong}
\address{Ping Zhong: Department of Pure Mathematics\\ University of Waterloo\\ 200 University Avenue West\\
Waterloo, ON, N2L 3G1, Canada.}
\email{ping.zhong@uwaterloo.ca}
\maketitle
\begin{abstract}
In this paper, we develop the notion of free-Boolean independence in an amalgamation setting.  We construct free-Boolean cumulants and show that the vanishing of mixed free-Boolean cumulants is equivalent to our free-Boolean independence with amalgamation.  
We also provide a characterization of free-Boolean independence by conditions in terms of mixed moments.  
In addition, we study free-Boolean independence over a $C^*$-algebra and prove a positivity property.
\end{abstract}

\section{Introduction}
Free probability theory is a probability theory that studies
noncommutative random variables with highest noncommutativity.  
This theory, due to Voiculescu, is based on the notion of free independence which  is an analogue of the classical independence.
In \cite{Voi3},  Voiculescu generalized his  notion of free independence to free independence with amalgamation over an arbitrary algebra in details.
To be specific,  moments of random variables are no longer scalar numbers but elements from a given algebra.   On the other aspect, Voiculescu started to 
study pairs of random variables  simultaneously thereby generalized the notion of free independence to a notion of bi-free independence \cite{Voi1}. 
Further more, the notion of bi-free independence with amalgamation, defined by Voiculescu \cite{Voi1}, was fully developed in \cite{CNS1}. 
there are exactly two unital universal independence relations, namely Voiculescu's free independence relation, the classical independence relation \cite{Sp}.
It was  mentioned that  we would obtain more independence relations by decreasing the number of axioms for universal products \cite{Liu3}.
 For instance, people introduced Boolean independence \cite{SW}, monotone independence \cite{Mu}, conditionally independence \cite{BLS} in various contexts.  
 Their operator-valued generalization were studied as well \cite{Ml,Po,Po1}.  
 Recently,  their corresponding independence relations for pairs of random variables, analog of Voiculescu's bi-free theory, were introduced and studied \cite{GS, GHS,  GS1}.  
Furthermore, the conditionally bi-free independence with amalgamation is studied in \cite{GS2}.

In \cite{Liu3},  the first-named author introduced a notion of mixed independence relations for pairs of random variables,
where random variables in different faces exhibit different kinds of noncommutative independence.
In particular, the combinatorics of free-Boolean independence relation was fully developed. 
In this paper, we generalize the notion of free-Boolean independence to an amalgamation setting. 
 Relevant combinatorial tools are extended to study this new independence.
Beyond the corresponding combinatorial results, we address the positivity of 
free-Boolean independence with amalgamation. 
Therefore, it is possible to study the relation in  topological probability spaces but not only algebraic probability spaces. 
For instance, we can study our free-Boolean independence with amalgamation over a $C^*$-algebra, 
which is a suitable framework to address some probabilistic questions.
We plan to study probabilistic results such as operator-valued infinitely divisible laws in a forthcoming paper.

The paper is organized as follows: 
Besides this introduction, in Section 2, we give the definition of free-Boolean 
independence with amalgamation over an algebra.
In Section 3, we review some relevant combinatorial tools. 
In Section 4, we demonstrate that free-Boolean independence can be 
characterized by the property of the vanishing of mixed free-Boolean cumulants.
In Section 5, we prove an operator-valued version of free-Boolean central limit law. 
In Section 6, we provide an equivalent characterization of free-Boolean independence 
by certain moments-conditions.
In Section 7,  we study the  positivity property for free-Boolean independence relation.

\section{Preliminaries and Notations}
In this section, we give the motivation and the definition for free-Boolean independence relation with amalgamation over an algebra.
\begin{definition}\normalfont
	A $\B$-$\B$-bimodule with a specified 
	projection is a triple $(\X, \mrx, p)$, where $\X$ is a direct sum of $\B$-$\B$-bimodules
	$\X = B \oplus \mrx$,
	and $p : \X \to B$ is the projection
	\[
	p(b \oplus \eta) = b.
	\] 
	Denote by $\LL(\X)$ the algebra of linear operators with respect 
	to the $\B$-$\B$-bimodule structure. The \emph{expectation}
	from $\LL(\X)$ onto $\B$ is the linear map
	$\E_{\LL(\X)}:\LL(\X)\rightarrow\B$ defined by
	\[
	\E_{\LL(\X)}(a)=p(a(1_{\B}\oplus 0)). 
	\]
\end{definition}

We now recall the definition of the reduced free product 
of $\B$-$\B$-bimodules with specified projections \cite{Sp1, Voi3}.
Let $\{(\X_i, \mathring{\X}_i, p_i)\}_{i \in \I}$ be a family of  $\B$-$\B$-bimodules with specified projections.  
The \emph{reduce free product of $\{(\X_k, \mathring{\X}_i, p_i)\}_{i \in \I}$ with amalgamation over $\B$} is defined to be the $\B$-$\B$-bimodule with a specified projection $(\X, \mrx, p)$,  where 
$\X = \B \oplus \mrx$ and $\mrx$ is the $\B$-$\B$-bimodule defined by
\begin{align*}
\mrx= \bigoplus\limits_{n\geq 1}\bigoplus_{i_1\neq i_2\neq\cdots\neq i_n} \mrx_{i_1}\otimes_B\cdots\otimes_B\mrx_{i_n}.
\end{align*}
For each $i\in\I$, we denote by
\[
\X(i)=\B\oplus\bigoplus_{n\geq 1}\bigoplus_{\substack{i_1\neq i_2\neq \cdots \neq i_n\\i_1\neq i}}
\mrx_{i_1}\otimes_B\cdots\otimes_B\mrx_{i_n},
\]
and let $V_i$ be the natural isomorphism of bimodules
\begin{align*}
V_i :  \X   \to \X_i\otimes_\B \X(i).
\end{align*}
For each $i\in\I$,  $\lambda_i:\LL(\X_i)\rightarrow\LL(\X)$ is a unital homomorphism defined by
\[
\lambda_i(a)=V_i^{-1}(a\otimes I)V_i
\]
and $\beta_i:\LL(\X_i)\rightarrow\LL(\X)$ is a linear map defined by 
\[
\beta_i(a)=P_i\lambda_i(a)P_i,
\]
where $P_i:\X\rightarrow \B\oplus \mrx_i$ is the 
natural projection onto $\B\oplus \mrx_i$ and vanishes on the other direct summands.

\begin{proposition}\label{simple Boolean}
For any $a\in \LL(X_i)$, we have
$P_i\lambda_i(a)=\lambda_i(a)P_i$. 
\end{proposition}
\begin{proof}
Notice that the reduced free product $\X$ of $\B$-$\B$-bimodules with specified projections 
can be decomposed as 
\[
\X=(\B\oplus \mrx_i)\oplus\X_i',
\quad \text{where}\quad \X_i'= \bigoplus\limits_{j\neq i}\mrx_j
 \oplus \bigoplus_{n\geq 2}\bigoplus_{\substack{i_1\neq i_2\neq \cdots \neq i_n}}
 \mrx_{i_1}\otimes_B\cdots\otimes_B\mrx_{i_n}.
\]
The space $\B\oplus \mrx_i$ is invariant under $\lambda_i(a)$ for any $a\in\LL(\X_i)$. 
We can check directly that the space $\X_i'$ is also invariant under $\lambda_i(a)$ for any $a\in\LL(\X_i)$
by the definition of $\lambda_i$. Hence the result follows. 
\end{proof}

The preceding result implies the next corollary. 
\begin{corollary}\normalfont
	The map $\beta_i:\LL(\X_i)\rightarrow \LL(\X)$ is a homomorphism.
\end{corollary}

\begin{definition}\normalfont
	A $B$-valued probability space is a pair $(\A,\E)$ consisting of an algebra  $\A$ over $B$
	 and an $\B$-$\B$-bimodule map $\E:\A\rightarrow\B$, i.e. a linear map such that
	$$\E(b_1 a b_2)=b_1\E(a)b_2$$ for all $a\in\A$ and $b_1, b_2\in\B$.
\end{definition}

\begin{definition}\label{free-BooleanDef}\normalfont
	Let $(\A,\E)$ be  a $B$-valued probability space. A family of $\B$-faces of $\A$ is a family $\{(C_i,D_i)\}_{i\in \I}$
	of (not necessarily unital) subalgebras 
	of $\A$ such that $C_i, D_i$ are $\B$-$\B$-bimodules for each $i \in \I$.  The family of $\B$-faces $\{(C_i, D_i)\}_{i\in \I}$ is said to be free-Boolean with amalgamation over $\B$ if 
\begin{itemize}
\item $\{C_i|i\in \I\}$ are unital algebras  
\item  there are $\B$-$\B$-bimodules with specified projections $\{(\X_i, \mrx_i, p_{i})\}_{i\in \I}$ such that there are unital homomorphisms $\gamma_i: C_i\rightarrow \LL(\X_i)$, (not necessarily unital) homomorphisms $\delta_i: D_i\rightarrow \LL(\X_i)$,
\item Let $(\X, \mrx, p)$  be the reduce free product of $\{(\X_i, \mrx_i, p_{i})\}_{i\in \I}$, so that the joint 
	distribution of the family $\{(C_i,D_i)\}_{i\in\I}$ in $(\A,\E)$
	is equal to the joint distribution of the family of operators
	$\{(\lambda_i(\gamma_i(C_i)), \beta_i(\delta_i(D_i))\}_{i\in \I}$ in
	the probability space $(\LL(\X), \E_{\LL(\X)})$. 
	That is, 
	\begin{align*}
	\E_{\LL(\X)}\big(\lambda_{i_1}(\gamma_{i_1}(c_1))
	\beta_{i_1}(\delta_{i_1}(d_1))\cdots
	 \lambda_{i_n}(\gamma_{i_n}(c_n))
	\beta_{i_n}(\delta_{i_n}(d_n))  \big)
	=\E(c_1 d_1\cdots c_n d_n),	
	\end{align*}
	where $c_k\in C_{i_k}, d_k\in D_{i_k}, 1\leq k\leq n.$
\end{itemize}

\end{definition}

\section{ Interval-noncrossing partitions}

In this section, we  review some combinatorial tools which will be used 

to define operator-valued free-Boolean cumulants. We give a  characterization 
of free-Boolean independence with amalgamation thereby generalizes
results in \cite{Liu3} to the operator-valued framework. 
In noncommutative probability theory, non-crossing partitions are used in 
the combinatorics of free probability and the interval partitions are used 
in the combinatorics of Boolean independence. 
It turns out the partitions used in the combinatorics of free-Boolean independence  
are so-called interval-noncrossing partitions introduced in \cite{Liu3}. 
All results without proof in this section are taken from \cite{Liu3}. 
\subsection{Interval-noncrossing partitions}
Throughout this section, we let $n\in\mathbb{N}$, 
$\chi:\{1,\cdots,n\}\rightarrow \{\f, \bb\}$ and $\epsilon:\{1,\cdots, n\}\rightarrow \mathcal{I}$,
for some fixed index set $\mathcal{I}$.
We will denote by $[n]$ the set  $\{1,\cdots,n\}$ for $n\in\mathbb{N}$.

\begin{definition}\label{partition}\normalfont
	Let $S$ be a linearly ordered set.
	A partition $\pi$ of the set $S$ consists of a collection
	disjoint, nonempty sets $\{V_1,\cdots,V_p\}$ whose union  is $S$.  
	The sets $V_1,\cdots,V_p$ are called the blocks of $\pi$. 
	Given $v_1,v_2\in S$, we write $v_1\sim v_2$ if the two elements $v_1, v_2$ are in the same block.
	\begin{itemize}
		\item[1.] A partition $\pi$ is called \emph{noncrossing}
		if there is no  quadruple $(v_1, v_2, w_1, w_2)$ such that
		$ v_1< w_1<v_2<w_2$, $v_1, v_2\in V_s$, and $w_1, w_2\in V_t$,
		where $V_s, V_t$ are two disjoint blocks of $\pi$.
		The set of all noncrossing partitions of $[n]$ will be denoted by $NC(n)$.  
		\item[2.] A block $V_s$ of $\pi$ is called 
		\emph{interval} if for any $v_1, v_2\in V_s$ and $v_1<w<v_2$, we have
		$w\in V_s$. 
		A partition $\pi=\{V_1, \cdots, V_p\}$ is called an \emph{interval} partition
		if all blocks $V_s$ are interval blocks. 
		\item[3.] A block $V_s$ of a partition  $\pi$ is said to be inner if there is another block $V_t\in \pi$ and  $v_1,v_2\in V_t$ such that $v_1<w<v_2$ for all $w\in V_s$. A block is outer if it is not inner.
		\item[4.] Let $\epsilon:[n]\rightarrow \mathcal{I}$.  We denote by $ker(\epsilon)$ the partition whose blocks are the sets $\omega^{-1}(i), i\in \mathcal{I}$. 
		\item[5.] Given two partitions $\sigma$ and $\pi$, we say $\sigma\leq \pi$ if each block of $\sigma$ is contained in a block of $\pi$. This relation is called the reversed refinement order. 
		\item[6.] We denote by $0_n$ the partition of $[n]$ consists of $n$ blocks
		and by $1_n$ the partition of $[n]$ consists of exactly one block. 
	\end{itemize}
\end{definition}

\begin{definition} \normalfont
	Let $\chi:\{1,\cdots,n\}\rightarrow \{\f, \bb\}$. 
	A partition $\pi$ of $[n]$ is said to be \emph{interval-noncrossing} with respect to $\chi$ if $\pi$  is  noncrossing,  and $v_1, v_2, w$ are in the same block
	whenever $v_1<w<v_2$, $v_1\sim v_2$ and $\chi(w)=\bb$. 
	We denote by $INC(\chi)$ the set of all interval-noncrossing partitions of the set $\{1, 2,\cdots,n \}$ with respect to $\chi$.
\end{definition}

\begin{remark}\normalfont
	The set $INC(\chi)$ does not depend on the value of $\chi$ at $1$ and $n$. 
	In particular, when $\chi^{-1}(\bb)\cap [2,n-1]=\emptyset$, we have $INC(\chi)=NC(n)$.
\end{remark}

For example, let $\chi:\{1,\cdots, 7\}\rightarrow\{\f,\bb\}$ such that $\chi^{-1}(\bb)=\{2,6,7\}$. Given two noncrossing  $\pi_1=\{\{1,4\},\{2\},\{3\}, \{5,7\},\{6\}\}$ and 
$\pi_2=\{\{1,2,4\},\{3\},\{5,6,7\}\}$ of the set $\{1,\cdots, 7\}$, then  $\pi_1\not\in INC(\chi)$ and $\pi_2\in INC(\chi)$. In pictures below, we use \textquotedbl$\bullet$\textquotedbl to denote elements in 
$\chi^{-1}(\f)$ and \textquotedbl$\circ$\textquotedbl to denote elements in 
$\chi^{-1}(\bb)$. \\
\begin{picture}(120.00,42.00)(-30.00, 0.00)

\put(0.00,21.00){\line(0,1){6.00}}
\put(0.00,27.00){\line(1,0){18.00}}
\put(12.00,21.50){\line(0,1){3.50}}
\put(18.00,21.00){\line(0,1){6.00}}

\put(6.00,21.50){\line(0,1){3.50}}

\put(24.00,27.00){\line(1,0){12.00}}
\put(24.00,21.50){\line(0,1){5.50}}
\put(30.00,21.50){\line(0,1){3.5}}
\put(36.00,21.50){\line(0,1){5.50}}

\put(0.00,21.00){\circle*{1.00}}
\put(6.00,21.00){\circle{1.00}}
\put(12.00,21.00){\circle*{1.00}}
\put(18.00,21.00){\circle*{1.00}}
\put(24.00,21.00){\circle*{1.00}}
\put(30.00,21.00){\circle{1.00}}
\put(36.00,21.00){\circle{1.00}}

\put(-0.90,17.00) {\footnotesize 1}
\put(05.10,17.00) {\footnotesize 2}
\put(11.10,17.00) {\footnotesize 3}
\put(17.10,17.00) {\footnotesize 4}
\put(23.10,17.00) {\footnotesize 5}
\put(29.10,17.00) {\footnotesize 6}
\put(35.10,17.00) {\footnotesize 7}

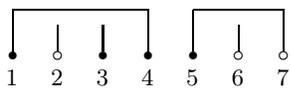
\put (6.00, 10.00){Diagram of $\pi_1$.}


\put(70.00,21.00){\line(0,1){6.00}}
\put(70.00,27.00){\line(1,0){18.00}}
\put(82.00,21.00){\line(0,1){4.00}}
\put(88.00,21.00){\line(0,1){6.00}}

\put(76.00,21.50){\line(0,1){5.50}}

\put(94.00,27.00){\line(1,0){12.00}}
\put(94.00,21.50){\line(0,1){5.50}}
\put(100.00, 21.50){\line(0,1){5.50}}
\put(106.00,21.50){\line(0,1){5.50}}

\put(70.00,21.00){\circle*{1.00}}
\put(76.00,21.00){\circle{1.00}}
\put(82.00,21.00){\circle*{1.00}}
\put(88.00,21.00){\circle*{1.00}}
\put(94.00,21.00){\circle*{1.00}}
\put(100.00, 21.00){\circle{1.00}}
\put(106.00, 21.00){\circle{1.00}}

\put(69.0,17.00) {\footnotesize 1}
\put(75.10,17.00) {\footnotesize 2}
\put(81.10,17.00) {\footnotesize 3}
\put(87.10,17.00) {\footnotesize 4}
\put(93.10,17.00) {\footnotesize 5}
\put(99.10,17.00) {\footnotesize 6}
\put(105.10,17.00) {\footnotesize 7}
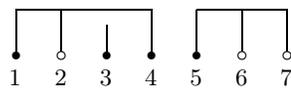
\put(76.00,10.00){Diagram of $\pi_2$.}

\end{picture}\\
Assume now $\chi^{-1}(\bb)\cap[2,n-1]=\{l_1<\cdots<l_{m-1}\}$
and set $l_0=0, l_m=n$. We denote by $[l_i, l_{i+1}]$ the
interval $\{l_i, l_i+1,\cdots, l_{i+1}\}$. 
For each $i=1,\cdots, m$, we denote by $\alpha_i(\pi)$ the restriction of $\pi$ to the interval
$[l_{i-1},l_{i}]$. 
We also denote by $\alpha'(\pi)$ the 
restriction of $\pi$ to the interval $[l_1,n]$
and $\chi'$ the restriction of $\chi$ 
to the interval $[l_1,n]$. Note that each $\alpha_i(\pi)$ can be any noncrossing partition of the set $[l_{i-1},l_{i}]$,
since there is no $l_{i-1}<w<l_i$ such that $\chi(w)=\bb$. 

\begin{proposition}\label{lattice isomorphism}\normalfont
	Let  $\alpha_1': INC(\chi)\rightarrow NC([1,l_1])\times INC (\chi')$ be
defined by
\[
\alpha_1'(\pi)=\left( \alpha_1(\pi), \alpha'(\pi)\right)
\]
and
$\alpha:INC(\chi)\rightarrow NC([1,l_1])\times NC([l_1,l_2])\times\cdots \times NC([l_{m-1},n])$ be defined by
	$$\alpha(\pi)=(\alpha_1(\pi), \cdots,\alpha_m(\pi)). $$
	Then $\alpha_1'$ and $\alpha$ are   isomorphisms of partial ordered sets. 
	The set $INC(\chi)$ is a lattice with respect to the reverse refinement order $\leq$ on partitions.
\end{proposition}
We provide pictures below to illustrate the preceding proposition. Let $n=10$, $\chi^{-1}(\bb)=\{1,3,7,8,9,10\}$ and $\pi=\{\{1,3,4,7\},\{2\},\{5,6\},\{9,8\},\{10\}\}$ which is an interval-noncrossing of the set $\{1,2,\cdots, 10 \}$ with respect to $\chi$ as shown in the following diagram.\\

\begin{picture}(90.00,30.00)(-50.00, 5.00)

\put(0.00,21.50){\line(0,1){7.50}}
\put(12.00,21.50){\line(0,1){7.50}}
\put(18.00,21.00){\line(0,1){8.00}}
\put(0.00,29.00){\line(1,0){36.00}}
\put(36.00,21.50){\line(0,1){7.50}}

\put(6.00,21.00){\line(0,1){4.00}}

\put(24.00,21.00){\line(0,1){4.00}}
\put(24.00,25.00){\line(1,0){6.00}}
\put(30.00,21.00){\line(0,1){4.00}}

\put(48.00,21.50){\line(0,1){7.50}}
\put(42.00,29.00){\line(1,0){6.00}}
\put(42.00,21.50){\line(0,1){7.50}}

\put(54.00,21.50){\line(0,1){7.50}}

\put(0.00,21.00){\circle{1.00}}
\put(6.00,21.00){\circle*{1.00}}
\put(12.00,21.00){\circle{1.00}}
\put(18.00,21.00){\circle*{1.00}}
\put(24.00,21.00){\circle*{1.00}}
\put(30.00,21.00){\circle*{1.00}}
\put(36.00,21.00){\circle{1.00}}
\put(42.00,21.00){\circle{1.00}}
\put(48.00,21.00){\circle{1.00}}
\put(54.00,21.00){\circle{1.00}}

\put(-0.90,17.00) {\footnotesize 1}
\put(05.10,17.00) {\footnotesize 2}
\put(11.10,17.00) {\footnotesize 3}
\put(17.10,17.00) {\footnotesize 4}
\put(23.10,17.00) {\footnotesize 5}
\put(29.10,17.00) {\footnotesize 6}
\put(35.10,17.00) {\footnotesize 7}
\put(41.10,17.00) {\footnotesize 8}
\put(47.10,17.00) {\footnotesize 9}
\put(53.10,17.00) {\footnotesize 10}

\put(15,9){Diagram of $\pi$.}
\end{picture}

In the above diagram, $l_0=1$, 
$l_1=3, l_2=7, l_3=8, l_4=9$, $l_5=10$. Therefore,
$\alpha_1(\pi)=\{\{1,3\},\{2\}\}$,  $\alpha_2(\pi)=\{\{3,4,7\},\{5,6\}\}$, $\alpha_3(\pi)=\{\{7\},\{8\}\}$, $\alpha_4(\pi)=\{\{8,9\}\}$ and $\alpha_5(\pi)=\{\{9\},\{10\}\}$ are illustrated
 in the following diagrams:\\

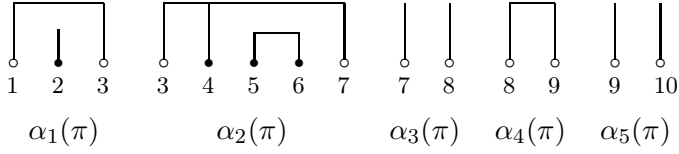
\begin{picture}(90.00,30.00)(-80.00, 5.00)

\put(-50.00,21.50){\line(0,1){7.50}}
\put(-38.00,21.50){\line(0,1){7.50}}
\put(-50.00,29.00){\line(1,0){12.00}}
\put(-44.00,21.50){\line(0,1){4.00}}

\put(-50.90,17.00) {\footnotesize 1}
\put(-44.90,17.00) {\footnotesize 2}
\put(-38.90,17.00) {\footnotesize 3}

\put(-50.00,21.00){\circle{1.00}}
\put(-44.00,21.00){\circle*{1.00}}
\put(-38.00,21.00){\circle{1.00}}

\put(-48,11){$\alpha_1(\pi)$}

\put(-30.00,21.50){\line(0,1){7.50}}
\put(-24.00,21.50){\line(0,1){7.50}}
\put(-30.00,29.00){\line(1,0){24.00}}
\put(-6.00,21.50){\line(0,1){7.50}}

\put(-18.00,21.00){\line(0,1){4.00}}
\put(-18.00,25.00){\line(1,0){6.00}}
\put(-12.00,21.00){\line(0,1){4.00}}

\put(-30.9,17.00) {\footnotesize 3}
\put(-24.9,17.00) {\footnotesize 4}
\put(-18.9,17.00) {\footnotesize 5}
\put(-12.9,17.00) {\footnotesize 6}
\put(-6.9,17.00) {\footnotesize 7}

\put(-30.00,21.00){\circle{1.00}}
\put(-24.00,21.00){\circle*{1.00}}
\put(-18.00,21.00){\circle*{1.00}}
\put(-12.00,21.00){\circle*{1.00}}
\put(-6.00,21.00){\circle{1.00}}

\put(-23,11){$\alpha_2(\pi)$}

\put(8.00,21.50){\line(0,1){7.50}}
\put(2.00,21.50){\line(0,1){7.50}}

\put(2.00,21.00){\circle{1.00}}
\put(8.00,21.00){\circle{1.00}}

\put(1.10,17.00) {\footnotesize 7}
\put(7.10,17.00) {\footnotesize 8}

\put(0,11){$\alpha_3(\pi)$}

\put(16.00,21.50){\line(0,1){7.50}}
\put(22.00,21.50){\line(0,1){7.50}}
\put(16.00,29.00){\line(1,0){6.00}}

\put(22.00,21.00){\circle{1.00}}
\put(16.00,21.00){\circle{1.00}}

\put(21.10,17.00) {\footnotesize 9}
\put(15.10,17.00) {\footnotesize 8}

\put(14,11){$\alpha_4(\pi)$}

\put(30.00,21.50){\line(0,1){7.50}}
\put(36.00,21.50){\line(0,1){7.50}}

\put(30.00,21.00){\circle{1.00}}
\put(36.00,21.00){\circle{1.00}}

\put(29.10,17.00) {\footnotesize 9}
\put(35.10,17.00) {\footnotesize 10}

\put(28,11){$\alpha_5(\pi)$}

\end{picture}

\begin{proposition}\normalfont
	Let $\pi=\{V_1,\cdots,V_p\}\in INC(\chi)$ and let 
	$\sigma\in NC(n)$ such that $\sigma\leq \pi$, i.e. each
	block of $\sigma$ is contained in a block of $\pi$. 
	Then $\sigma\in INC(\chi)$ if and only if
	$\sigma|_{V_s}\in INC(\chi|_{V_s})$ for all $s=1,\cdots, p$. 
\end{proposition}

\begin{proposition}\label{canonical isomorphism}	\normalfont
	Let $\pi=\{V_1,\cdots,V_p\}\in INC(\chi)$.  Denote by $[0_n, \pi]$
	the set of all $\sigma\in INC(\chi)$ such that $\sigma\leq \pi$. 
	Then
	\[
	[0_n,\pi]\cong INC(\chi|_{V_1})\times\cdots\times  INC(\chi|_{V_p}). \
	\]
\end{proposition}

\subsection{M\"obius functions on interval-noncrossing partitions} 
One can define the convolution for functions 
on the lattice following the standard procedure for 
partially ordered sets (see \cite{Rota}). 
Once the map $\chi$ is fixed, the lattice structure
of $INC(\chi)$ caputred from the lattice of the product 
of noncrossing partitions according to the natural
isomorphism described in Proposition \ref{lattice isomorphism}. 

Let $\chi:\{1,\cdots, n\}\rightarrow\{\f, \bb\}$. 
Given two complex-valued functions defined
on the set $\{(\sigma,\pi)|\sigma, \pi\in INC(\chi), \sigma\leq \pi\}$.
The  convolution of $f$ and $g$ is given by
$$f*g(\sigma,\pi)=\sum\limits_{ \substack{\rho\in INC(\chi)\\ \sigma\leq\rho\leq\pi}} f(\sigma,\rho)g(\rho,\pi).$$
 The delta function defined as follows:
\[
\delta_{INC}(\sigma,\pi)=\left\{\begin{array}{ll}
1,\,\,\,\, &\text{if}\,\, \sigma=\pi,\\
0,&\text{otherwise.}
\end{array}\right.
\]
We then define the zeta function by 
\[
\zeta_{INC}(\sigma,\pi)=\left\{\begin{array}{ll}
1,\,\,\,\, &\text{if}\,\, \sigma\leq\pi,\\
0,&\text{otherwise.}
\end{array}\right.
\]
and the M\"obius function $\mu_{INC}$ is the 
inverse of the zeta function in the following sense: 
\[
\mu_{INC}*\zeta_{INC}=\zeta_{INC}*\mu_{INC}=\delta_{INC}.
\]

We will use the following product formula in  \cite[Section 6]{Liu}.
\begin{proposition}\label{M product formula}\normalfont
	Let $\chi:\{1,\cdots, n\}\rightarrow \{\f, \bb \}$,
	$\pi=\{V_1,\cdots,V_p\}\in INC(\chi)$ 
	and $\sigma\in INC(\chi)$ such that $\sigma\leq \pi$.
	Suppose that  $l_0=1, l_m=n$ and $\chi^{-1}(\bb)\cap[2,n-1]=\{l_1<\cdots<l_{m-1}\}$,
	then 
	\[
	\begin{array}{rcl}
	\mu_{INC(\chi)}(\sigma,\pi)&=&\prod\limits_{i=1}^m \mu_{INC(\chi|_{V_i})}(\sigma_i,\pi_i)\\
	&=&\prod\limits_{s=1}^p\mu_{INC}(\sigma|_{V_s},1_{V_s}) \\
	&=&\prod\limits_{i=1}^m\prod\limits_{s=1}^{p} \mu_{NC}(\sigma_i|\talpha_i(V_s),1_{\talpha_i(V_s)}),
	\end{array}\\
	\]
where $\sigma_i=\alpha_i(\sigma)$, $\pi_i=\alpha_i(\pi)$, $1\leq i\leq m$ and 
$\talpha_i(V_s)$ is the restriction of $V_s$ to the set $[l_{i-1}, l_i]$.
\end{proposition}

\begin{corollary}\label{one canonical isomorphism}\normalfont 
	Let $\pi\in INC(\chi)$ and $V\in\pi$.  
	Denote by $[0_n,\pi]=\{\sigma\in INC(\chi): 0_n\leq \sigma\leq \pi\}$,
	and $V'=\{1,\cdots, n\}\backslash V$. Then,
	$$[0_n, \pi]\cong INC(\chi|_{V})\times INC(\chi|{V'}).$$
	In particular, we have
	\[
	   \mu_{INC}(\sigma, \pi)=\mu_{INC}(\sigma|_V, 1_V)\mu_{INC}(\sigma|_{V'}, \pi|_{V'}) 
	\]
	for $\sigma\in INC(\chi)$ and $\sigma\leq \pi$.
\end{corollary}

\section{ Vanishing cumulants condition for free-Boolean independence}
In this section, we introduce the notion of operator-valued free-Boolean cumulants for pairs of random variables and  
give an alternative characterization of free-Boolean independence by using 
the free-Boolean cumulants. 

\subsection{Free-Boolean cumulants}
Let $(\A,\E)$ be a $B$-valued  probability space . Let $\Phi^{(n)}$ be the $n$-B-linear map from $\underbrace{\A\otimes_\B\cdots\otimes _\B\A}_{\text{n times}}$ to $\B$  defined as
$$\Phi^{(n)}(a_1,\cdots,a_n)=\E(a_1\cdots a_n).$$
Then, for each noncrossing partition $\pi\in NC(n)$,  we 
can write $\pi=\pi_1\cup \{V\}$, where 
$V=(l+1,l+2,\cdots,l+s)$ is an interval block of $\pi$
and $\pi_1=\pi\setminus \{V\}$. 
We define an $n$-B-linear map $\Phi_\pi: \underbrace{\A\otimes_\B\cdots\otimes_\B \A}_{\text{n times}}\rightarrow \B$ recursively as follows:
\begin{equation}\label{def:n-B-map}
\Phi_\pi(a_1,\cdots,a_n)=\Phi_{\pi_1}(a_1,\cdots,a_l,\Phi^{(s)}(a_{l+1},\cdots,a_{l+s})a_{l+s+1},\cdots,a_n).
\end{equation}

For example,  let $\pi=\{\{1,5,8\},\{2,3,4\},\{6,7\}\}$ be a noncrossing
partition of $\{1,\cdots, 8\}$. Then,
$$\Phi_{\pi}(a_1,\cdots,a_8)=\Phi_{\{\{1,5,8\},\{6,7\}\}}(a_1,\E(a_2a_3a_4)a_5,a_6,a_7,a_8)=\E(a_1\E(a_2a_3a_4)a_5\E(a_6a_7)a_8).$$
\begin{definition}  \normalfont
	Given any $\chi:\{1,\cdots, n\}\rightarrow\{\f,\bb\}$, $\pi\in INC(\chi)$
and a tuple of elements $(a_1,\cdots, a_n)$ in $(\A, \E)$, the free-Boolean cumulant $\kappa_{\chi,\pi}$ is an $n$-$B$-linear map defined as follows:
	$$\kappa_{\chi,\pi}(a_1,\cdots,a_n)
	=\sum\limits_{\substack{\sigma\leq \pi\\ \sigma\in INC(\chi)}}\mu_{INC}(\sigma,\pi)\Phi_{\sigma}(a_1,\cdots,a_n) .$$ 
\end{definition}
We start to show that the operator-valued free-Boolean cumulants have the following  multiplicative property.
\begin{theorem}\label{multiplicative property}\normalfont 
	Let $\pi\in INC(\chi)$ and $a_1,\cdots,a_n$ be noncommutative random variables in  a $B$-valued probability space $(\A, \E)$. Suppose that $V=(l+1,l+2,\cdots,l+s)$ is an interval block of $\pi$, then
	$$
	\begin{array}{rcl}
	\kappa_{\chi,\pi}(a_1,\cdots,a_n)&=&\kappa_{\chi|_{V'},\pi|_{V'}}\big(a_1,\cdots,a_l,\kappa_{\chi|_{V},1_{V}}(a_{l+1},\cdots,a_{l+s})a_{l+s+1},\cdots,a_n\big)\\
	  &=&\kappa_{\chi|_{V'},\pi|_{V'}}\big(a_1,\cdots, (a_l\kappa_{\chi|_{V},1_{V}}(a_{l+1},\cdots,a_{l+s})), a_{l+s+1},\cdots,a_n\big),
		\end{array}
	$$ 
	where $V'=\{1,\cdots, n\}\setminus V.$
\end{theorem}
\begin{proof}
	For any $\sigma\leq \pi$, $\sigma\in INC(\chi)$, one can decompose it
	into a union of  two interval-noncrossing partitions $\sigma=\sigma_1\cup \sigma_2$, 
	where $\sigma_1\leq \pi|_{V'}$, $\sigma_1\in INC(\chi|_{V'})$ and $\sigma_2\leq \pi|_V=1_V$,
	 $\sigma_2\in INC(\chi|_V)$. 
    By Proposition \ref{M product formula} and Corollary \ref{one canonical isomorphism}, we  have
	$$
	\begin{array}{rcl}
	& &\kappa_{\chi,\pi}(a_1,\cdots,a_n)
	
	=\sum\limits_{\substack{\sigma\leq \pi\\ \sigma\in INC(\chi)}}\mu_{INC}(\sigma,\pi)\Phi_{\sigma}(a_1,\cdots,a_n)\\
	
	&=&\sum\limits_{\substack{\sigma\leq \pi\\ \sigma\in INC(\chi)}}\mu_{INC}(\sigma,\pi)\Phi_{\sigma|_{V'}}\big(a_1,\cdots,a_l,\Phi_{\sigma|_V}(a_{l+1},\cdots,a_{l+s})a_{l+s+1},\cdots,a_n\big) \\
	
	
	&=&\sum\limits_{\substack{\sigma\leq \pi\\ \sigma\in INC(\chi)}}\mu_{INC}(\sigma|_{V'},\pi|_{V'})\mu_{INC}(\sigma|_V,1_{V})\Phi_{\sigma|_{V'}}
	\big(a_1,\cdots,a_l,\Phi_{\sigma|_V}(a_{l+1},\cdots,a_{l+s})a_{l+s+1},\cdots,a_n\big) \\
	
	
	&=&\sum\limits_{\substack{\sigma_1\leq \pi|_{V'}\\ \sigma_1\in INC(\chi|_{V'})\\ \sigma_2\in INC(\chi|_{V})}}
	\mu_{INC}(\sigma_1,\pi|_{V'})
	\mu_{INC}(\sigma_2,1_{V})
	\Phi_{\sigma|_{V'}}\big(a_1,\cdots,a_l,\Phi_{\sigma|_V}(a_{l+1},\cdots,a_{l+s})a_{l+s+1},\cdots,a_n\big) \\
	
	&=&\sum\limits_{\substack{\sigma_1\leq \pi|_{V'}\\ \sigma_1\in INC(\chi|_{V'})}}\mu_{INC}(\sigma_1,\pi|_{V'})\Phi_{\sigma|_{V'}}\big(a_1,\cdots,a_l, \\
   & &\quad \quad \quad \quad \quad \quad \quad \quad \quad	\sum\limits_{\sigma_2\in INC(\chi|_{V})}\mu_{INC}(\sigma_2,1_{V})\Phi_{\sigma|_V}(a_{l+1},\cdots,a_{l+s})a_{l+s+1},\cdots,a_n\big) \\
	
	&=&\sum\limits_{\substack{\sigma_1\leq \pi|_{V'}\\ \sigma_1\in INC(\chi|_{V'})}}\mu_{INC}(\sigma_1,\pi|_{V'})\Phi_{\sigma|_{V'}}\big(a_1,\cdots,a_l,
	\kappa_{\chi|_{V},1_{V}}(a_{l+1},\cdots,a_{l+s})a_{l+s+1},\cdots,a_n\big) \\

	&=&\kappa_{\chi|_{V'},\pi|_{V'}}(a_1,\cdots,a_l,\kappa_{\chi|_{V},1_{V}}(a_{l+1},\cdots,a_{l+s})a_{l+s+1},\cdots,a_n).\\
	\end{array}
	$$
The other part follows from the bi-module property. 
This finishes the proof.
\end{proof}
The preceding theorem shows that $\kappa_{\chi,\pi}(a_1,\cdots,a_n)$ is completely determined by cumulant functionals of the form $\kappa_{\chi',1_{[m]}}$ for 
$m\in\mathbb{N}$ and $\chi':\{1,\cdots,m \}\rightarrow \{\f,\bb \}$.

\begin{definition}\label{FB cumlant definition}\normalfont
	Let $\{(\mathcal{A}_{i,\f}, \mathcal{A}_{i,\bb})\}_{i \in \mathcal{I}}$ be a family of pairs of $B$-faces of $\A$ in a $B$-valued probability space $(\A, \E)$. We say that the family $\{(\mathcal{A}_{i,\f}, \mathcal{A}_{i,\bb})\}_{i \in \mathcal{I}}$ is combinatorially free-Boolean independent with amalgamation over $\B$ if 
	$$\kappa_{\chi,1_n}(a_1,\cdots, a_n)=0 $$
	whenever  $\omega : \{1,\cdots.,n\}\to \I$, $\chi:\{1,\cdots, n\}\to \{\f, \bb\}$, $a_k\in\A_{\omega(k),\chi(k)}$ and 
	$\omega$ is not a constant.
\end{definition}

\begin{proposition}\normalfont
	Let $\{(\mathcal{A}_{i,\f}, \mathcal{A}_{i,\bb})\}_{i \in \mathcal{I}}$ be a family of pairs of $B$-faces in a $B$-valued probability space $(\A, \E)$. Then $\kappa_{\chi,1_n}$ has the following  additivity property:
	$$ \kappa_{\chi,1_n}(a_{1,1}+a_{2,1},\cdots, a_{1,n}+a_{2,n})=\kappa_{\chi,1_n}(a_{1,1},\cdots, a_{1,n})+\kappa_{\chi,1_n}(a_{2,1},\cdots, a_{2,n})$$
	whenever $\omega_1,\omega_2:[n]=\{1,\cdots, n\}\rightarrow \I$,
	$\chi:\{1,\cdots,n\}\to \{\f,\bb\}$, $a_{1,k}\in\A_{\omega_1(k),\chi(k)}$, $a_{2,k}\in\A_{\omega_2(k),\chi(k)}$  and $\omega_1([n])\cap\omega_2([n])=\emptyset$.
\end{proposition}
\begin{proof}
	By a direct calculation, we have
	$$\kappa_{\chi,1_n}(a_{1,1}+a_{2,1},\cdots, a_{1,n}+a_{2,n})=\sum\limits_{i_1,...i_n\in\{1,2\}}\kappa_{\chi,1_n}(a_{i_1,1},\cdots, a_{i_n,n}).$$
	Since $\{(\mathcal{A}_{i,\f}, \mathcal{A}_{i,\bb})\}_{i \in I}$ are combinatorially free-Boolean independent, by Definition \ref{FB cumlant definition}, we have 
	$$\kappa_{\chi,1_n}(a_{i_1,1},\cdots, a_{i_n,n})=0$$
	if $i_j\neq i_k$ for some $j,k\in\{1,\cdots, n\}$. The result follows.
\end{proof}

\begin{proposition}\label{vanishing cumu}\normalfont
	Let $\{(\mathcal{A}_{i,\f}, \mathcal{A}_{i,\bb})\}_{i \in \mathcal{I}}$ be a combinatorially free-Boolean independent family of pairs of $B$-faces in a $B$-valued probability space $(\A, \E)$. Assume that $\pi=\{V_1,\cdots,V_p\}\in INC(\chi)$.  Then
	$$\kappa_{\chi,\pi}(a_1,\cdots, a_n)=0 $$
	whenever  $\omega : \{1,\cdots.,n\}\to \I$, 
	$\chi:\{1,\cdots, n\}\to \{\f,\bb \}$, $a_k\in\A_{\omega(k),\chi(k)}$ and 
	$\omega$ is not a constant on  a block $W$ of $\pi$.
\end{proposition}
\begin{proof} We prove the statement by induction on the number of blocks of $\pi$.
	
	When $p=1$, then the statement follows from Definition \ref{FB cumlant definition}.	
	Suppose now that $p>1$, let $V=(l+1,l+2,\cdots,l+s)$ be an interval block of $\pi$.  By Proposition \ref{multiplicative property}, we have
	$$\kappa_{\chi,\pi}(a_1,\cdots,a_n)=\kappa_{\chi|_{V'},\pi|_{V'}}(a_1,\cdots,a_l,\kappa_{\chi|_{V},1_{V}}(a_{l+1},\cdots,a_{l+s})a_{l+s+1},\cdots,a_n),$$ 
	where $V'=\{1,\cdots, n\}\setminus V.$ If $\omega$ is not a constant on $V$, then $\kappa_{\chi,\pi}(a_1,\cdots, a_n)=0$.  Otherwise, $\omega|_{V'}$ in not a constant on a block of $\pi|_{V'}.$ The statement follows from an induction argument.
\end{proof}

\subsection{Free-Boolean independence is equivalent to combinatorially free-Boolean independence }
In this subsection,  we will prove that free-Boolean independence 
defined in Definition \ref{free-BooleanDef}
is equivalent to the combinatorially free-Boolean independence
given in Definition \ref{FB cumlant definition}.
We will show that mixed  moments  are uniquely determined by lower order mixed moments in the same way for both free-Boolean independence and  combinatorially free-Boolean independence.  

The proof for following result is essentially the same
as the proof of  in \cite[Proposition 10.6]{NS} in free probability context and we thus leave the details
to the reader. 
Applying Theorem \ref{multiplicative property}, we have  the following result.

\begin{lemma}\label{Moments-cumulant}   \normalfont
	Let $\chi:\{1,\cdots,n\}\to \{\f, \bb \}$ and $a_1,\cdots,a_n$ be noncommutative random variables in a $B$-valued probability space $(\A, \E)$.  Then
	$$\E(a_1\cdots a_n)=\sum\limits_{\pi\in INC(\chi)} \kappa_{\chi,\pi} (a_1\cdots a_n).$$
\end{lemma}

For combinatorially free-Boolean independent random variables, we have the following result.
\begin{lemma}\label{c-fb}
	Let  $\{(\mathcal{A}_{i,\f}, \mathcal{A}_{i,\bb})\}_{i \in \mathcal{I}}$ be a family of  combinatorially free-Boolean independent  pairs of $B$-faces in  a $B$-valued probability space $(\A,\E)$.  Assume that $a_k\in\A_{\omega(k),\chi(k)}$, where $\omega:\{1,\cdots, n\}\rightarrow \I$, $\chi:\{1,\cdots, n\}\rightarrow \{\f,\bb\}$.  Let $\epsilon=\ker (\omega)$. Then, 
	\begin{equation}\label{recursive relation}
	\E(a_1\cdots a_n)=\sum\limits_{\sigma\in INC(\chi)} \left(\sum\limits_{\substack{ \pi\in INC(\chi)\\ \sigma\leq \pi\leq \epsilon}}\mu_{INC}(\sigma,\pi)\right)\Phi_{\sigma} (a_1\cdots a_n). \tag{$\bigstar$}
	\end{equation}
\end{lemma}
\begin{proof}
	By Lemma \ref{Moments-cumulant}, we have 
	$$\E(a_1\cdots a_n)=\sum\limits_{\pi\in INC(\chi)} \kappa_{\chi,\pi} (a_1\cdots a_n).$$
	For each $\pi\in INC(\chi)$, write its blocks as $\pi=\{V_1,\cdots,V_p\}$.  Since $\{(\mathcal{A}_{i,\f}, \mathcal{A}_{i,r})\}_{i \in \mathcal{I}}$ are combinatorially free-Boolean independent, 
	by Lemma \ref{vanishing cumu}, we have
	$$\kappa_{\chi,\pi} (a_1\cdots a_n)=0,$$
	if $\omega$ is not a  constant on some block $V_s$ of $\pi$.   
	In other words, $\kappa_{\chi,\pi}(a_1,\cdots,a_n)\neq 0$ 
	only if $\omega$ is  a constant on ${V_s}$ for all $s$, 
	which implies that $V_s$ is contained in a block of $\epsilon$ for all $s$, i.e.,  $\pi\leq \epsilon$.
	Therefore, we have
	$$
	\begin{array}{rcl}
	\E(a_1\cdots a_n)
	&=&\sum\limits_{\pi\in INC(\chi),\pi\leq \epsilon} \kappa_{\chi,\pi} (a_1,\cdots ,a_n)\\
	&=&\sum\limits_{\pi\in INC(\chi),\pi\leq \epsilon} \left(\sum\limits_{\substack{ \sigma\in INC(\chi)\\ \sigma\leq \pi}}\mu_{INC}(\sigma,\pi)\Phi_{\sigma} (a_1,\cdots, a_n)\right)\\
	&=&\sum\limits_{\sigma\in INC(\chi)} \left(\sum\limits_{\substack{ \pi\in INC(\chi)\\ \sigma\leq \pi\leq \epsilon}}\mu_{INC}(\sigma,\pi)\right)\Phi_{\sigma} (a_1,\cdots ,a_n).\\
	\end{array}
	$$
	This finishes the proof. 
\end{proof}

We now turn to consider the case that the family $\{(\mathcal{A}_{i,\f}, \mathcal{A}_{i,\bb})\}_{i \in \mathcal{I}}$ is  free-Boolean independent in $(\A,\B,\E)$ in the sense of Definition \ref{free-BooleanDef}. 
In what follows, we assume that $a_k\in\A_{\omega(k),\chi(k)}$, where $\omega:\{1,\cdots, n\}\rightarrow \I$, $\chi:\{1,\cdots, n\}\rightarrow \{\f, \bb\}$. 
Let $\epsilon=\text{ker}(\omega)$, the kernel of $\omega$. 
Recall that $\chi^{-1}(\bb)\cap [2,n-1]=\{l_1, \cdots, l_{m-1} \}$.
Let $\chi_1$ (or $\epsilon_1$) be the restriction of $\chi$ (or $\epsilon$)  to $\{1,\cdots, l_1\}$
respectively.
Let $\chi_1'$ (or $\epsilon_1'$) be the restriction of 
$\chi$ (or $\epsilon$) to the interval $\{l_1,\cdots, n\}$ respectively.  
We need to show that  the the mixed moments $\E(a_1\cdots a_n)$ can be determined in the same way as in Lemma \ref{c-fb}. 

To this end, it is enough to consider the case that $\A=\LL(\X)$,  $\mathcal{A}_{i,\f} =\lambda_i(\mathcal{L}(\X_i))$ and $\mathcal{A}_{i,\bb} =P_i\lambda_i(\mathcal{L}(\X_i))P_i$, where   $\{(\X_i, \mrx_i,p_i)\}_{i\in I}$ is a family of vector spaces with specified vectors and $(\mathcal{X}, \mrx, p)$ is the  reduced free product  of them.

We will prove the mixed moments formula $(\bigstar)$ in Lemma \ref{c-fb}   by  induction on the number of elements  of  $\chi^{-1}(\bb)\cap[2,n-1]$.

\begin{lemma}\label{replace a_n}  If $\chi(n)=\bb$, then there exists an operator $T\in \A_{\omega(n),\f}$ such that 
	$$\E(a_1\cdots a_n)= \E(a_1\cdots a_{n-1}T).$$
\end{lemma}
\begin{proof}
	If $n\in \chi^{-1}(\bb)$, then $a_n\in \A_{\omega(n),\bb}=P_{\omega(n)}\lambda_{\omega(n)}(\LL(\X_{\omega(n)}))P_{\omega(n)}$. Assume that $a_n=P_{\omega(n)}TP_{\omega(n)}$ for some $T\in \lambda_{\omega(n)}(\mathcal{L}(\X_{\omega(n)}))$. Then 
	$$a_1\cdots a_n1_{\B}=a_1\cdots a_{n-1}P_{\omega(n)}TP_{\omega(n)}1_{\B}=a_1\cdots T1_{\B}=a_1\cdots a_{n-1}T1_{\B}$$
	since $ T1_\B\in P_{\omega(n)}\X$.  Thus,  the mixed moments are the same if we replace $a_n$ by the element $T\in\lambda_{\omega(n)}(\mathcal{L}(\X_{\omega(n)}))$.
\end{proof}

\begin{lemma}\label{replace a_1} If $\chi(1)=\bb$, then there exists an operator $T\in \A_{\omega(n),\f}$ such that 
	$$\E(a_1\cdots a_n)= \E(Ta_2\cdots a_{n}).$$
\end{lemma}
\begin{proof}
	If $1\in \chi^{-1}(\bb)$, then $a_1\in \A_{\omega(1),\bb}=P_{\omega(1)}\lambda_{\omega(1)}(\mathcal{L}(\X_{\omega(1)}))P_{\omega(1)}$. Assume that $a_1=P_{\omega(1)}TP_{\omega(1)}$ for some $T\in \lambda_{\omega(1)}(\mathcal{L}(\X_{\omega(1)}))$.  
	Recall that $p$ is the projection $p:\X\rightarrow B$. 
	Notice that  $pP_{\omega(1)}=p$ and 
	
	$$
	\begin{array}{rcl}
	\E(a_1\cdots a_n)
	&=&pP_{\omega(1)}TP_{\omega(1)}a_2\cdots a_{n}1_{\B}\\
	&=&pP_{\omega(1)}TP_{\omega(1)}a_2\cdots a_{n}1_{\B}\\
	&=&pTP_{\omega(1)}a_2\cdots a_{n}1_{\B}\\
	
	&=&pTa_2\cdots a_{n}1_{\B}-pT(1_{\X}-P_{\omega(1)})a_2\cdots a_{n}1_{\B},\\
	\end{array}$$
	where $1_{\X}$ is the identity operator in $\LL(\X)$. Notice that  
	$$(1_{\X}-P_{\omega(1)})a_2\cdots a_{n}1_{\B}\in \bigoplus\limits_{i\neq\omega(1)}\mrx_i\oplus\bigoplus\limits_{n\geq 2}\left(\bigoplus\limits_{i_1\neq i_2\neq\cdots \neq i_n} \mathring{\mathcal{X}_{i_1}}\otimes_{\B}\cdots \otimes_{\B}\mathring{\mathcal{X}_{i_n}}\right) $$
	and 
\[
	\begin{array}{crl}
	&&\bigoplus\limits_{i\neq\omega(1)}\mrx_i\oplus\bigoplus\limits_{n\geq 2}\left(\bigoplus\limits_{i_1\neq i_2\neq\cdots \neq i_n} \mathring{\mathcal{X}_{i_1}}\otimes_{\B}\cdots \otimes_{\B}\mathring{\mathcal{X}_{i_n}}\right)\\
	 &=&V_{\omega(1)}\left(\X_{\omega(1)}\otimes_{\B}  
	\left(\bigoplus\limits_{ \substack{ i_1\neq i_2\neq\cdots \neq i_n\\ i_1 \neq \omega(1), n\geq 1} }
	\mrx_{i_1}\otimes_{\B}\cdots \otimes_{\B}\mathring{\mathcal{X}_{i_n}}\right) \right).
	\end{array}
\]
	Therefore, $	\bigoplus\limits_{i\neq\omega(1)}\mrx_i\oplus\bigoplus\limits_{n\geq 2}\left(\bigoplus\limits_{i_1\neq i_2\neq\cdots \neq i_n} \mathring{\mathcal{X}_{i_1}}\otimes_{\B}\cdots \otimes_{\B}\mathring{\mathcal{X}_{i_n}}\right)$ is an invariant subspace of $T$ and 
	$$pT(I_{\X}-P_{\omega(1)})a_2\cdots a_{n}1_{\B}=p(I_{\X}-P_{\omega(1)})T(I_{\X}-P_{\omega(1)})a_2\cdots a_{n}1_{\B})=0, $$
	where the last equality follows from the fact that $p(I_{\X}-P_{\omega(1)})=0.$ We thus proved that the mixed moments
	$	\E(a_1\cdots a_n)$ will be the same if we replace $a_1$ by the element $T\in\lambda_{\omega(1)}(\mathcal{L}(\X_{\omega(1)}))$.
	
\end{proof}

We start with the following result. 
\begin{lemma} \label{lemma:4.10}
When $|\chi^{-1}(\bb)\cap[2,n-1]|=0$, we have 
	$$\E(a_1\cdots a_n)=\sum\limits_{\sigma\in INC(\chi)} \left(\sum\limits_{\substack{ \pi\in INC(\chi)\\ \sigma\leq \pi\leq \epsilon}}\mu_{INC}(\sigma,\pi)\right)\Phi_{\sigma} (a_1\cdots a_n).\\$$
\end{lemma}
\begin{proof}
	By Lemma \ref{replace a_n} and \ref{replace a_1}, 
	if $a_1\in \A_{\omega(1), \bb}$ or $a_n\in\A_{\omega(n), \bb}$, we may replace $a_1$ by $T_1\in \A_{\omega(1), \f}$
	and $a_n$ by $T_2\in\A_{\omega(n), \f}$, we still have
	$$ \Phi(a_1\cdots a_n)=\Phi(T_1a_2\cdots a_{n-1}T_2).$$
	Hence, when $|\chi^{-1}(\bb)\cap[2,n-1]|=0$, we may assume that $T_1, a_2,\cdots,a_{n-1}, T_2$ are from the left faces of algebras $\A_{\omega(k), \f}$. Notice that 
	the family $\{(\mathcal{A}_{i,\f})\}_{i \in \mathcal{I}}$ is  freely independent with amalgamation in $(\A,\E)$ (see \cite{NS, Sp1}), we have $$
	\begin{array}{rcl}
	\E(a_1\cdots a_n)&=&\E(T_1a_2\cdots a_{n-1}T_2)\\
	&=&\sum\limits_{\pi\in NC(n),\pi\leq \epsilon} \kappa_{\pi} (T_1,a_2,\cdots ,a_{n-1},T_2)\\
	&=&\sum\limits_{\sigma\in NC(n)} \left(\sum\limits_{\substack{ \pi\in NC(n)\\ \sigma\leq \pi\leq \epsilon}}\mu(\sigma,\pi)\right)\Phi_{\sigma} (T_1,a_2,\cdots ,a_{n-1},T_2)\\
	
	&=&\sum\limits_{\sigma\in INC(\chi)} \left(\sum\limits_{\substack{ \pi\in INC(\chi)\\ \sigma\leq \pi\leq \epsilon}}\mu_{INC}(\sigma,\pi)\right)\Phi_{\sigma} (T_1,a_2,\cdots ,a_{n-1},T_2)\\
		&=&\sum\limits_{\sigma\in INC(\chi)} \left(\sum\limits_{\substack{ \pi\in INC(\chi)\\ \sigma\leq \pi\leq \epsilon}}\mu_{INC}(\sigma,\pi)\right)\Phi_{\sigma} (a_1,a_2,\cdots ,a_{n-1},a_2),
	\end{array} 
	$$
	where we used the fact that $INC(\chi)=NC(n)$ when $|\chi^{-1}(\bb)\cap[2,n-1]|=0$.                             

\end{proof}

Now, we are now ready to prove our main theorem.
\begin{theorem}\label{Main}\normalfont
	Let $\{(\mathcal{A}_{i,\f}, \mathcal{A}_{i,\bb})\}_{i \in I}$ be a family of pairs of $B$-faces in a $B$-valued probability space $(\A, \E)$. 
	The family $\{(\mathcal{A}_{i,\f}, \mathcal{A}_{i,\bb})\}_{i \in \mathcal{I}}$ is free-Boolean independent with amalgamation over $B$ if and only if they are combinatorially free-Boolean independent with amalgamation over $B$.
\end{theorem}
\begin{proof}
    It suffices to show that Equation $(\bigstar)$ holds by assuming that $\{(\mathcal{A}_{i,\f}, \mathcal{A}_{i,\bb})\}_{i \in \mathcal{I}}$ is free-Boolean independent with amalgamation over $B$. When $|\chi^{-1}(\bb)\cap[2,n-1]|=0$, it is Lemma \ref{lemma:4.10}. 
	Assume now that  Equation $(\bigstar)$ in Lemma \ref{c-fb} holds  whenever $|\chi^{-1}(\bb)\cap[2,n-1]|
	\leq m-2$.  We shall prove it holds when $|\chi^{-1}(\bb)\cap[2,n-1]|=m-1$. Set $\chi^{-1}(\bb)=\{l_1<\cdots<l_{m-1}\}$ and $l_0=1, l_m=n$.
	
	Let $A_1=\prod\limits_{i=l_1}^n a_i$. Then $A_1(1_{\B})\in \B\oplus \mrx_{\omega(l_1)}$. Since the range of $A_{l_1}$ is  $\B\oplus \mrx_{\omega(l_1)}$, we can view $A_1:\B\oplus \mrx_{\omega(l_1)}\to \B\oplus \mrx_{\omega(l_1)}$ as a linear operator. In this way, $A_1$ is considered as an element in $\lambda_{\omega(l_1)}(\LL(\X_{\omega(l_1)}))$. 

	Apply the induction for the $l_1$-tuple $(a_1, \cdots, a_{l_1-1}, A_1)$. Recall that $\epsilon_1$ is the restriction
	of $\epsilon$ to the interval $[1, l_1]$, we have 
	\begin{equation}\label{eq:001}
	\begin{array}{rcl}
	\E(a_1\cdots a_n)&=&\E(a_1\cdots a_{l_1-1}A_1)\\
	&=&\sum\limits_{\substack{\sigma_1\in NC([l_1]) \\ \sigma_1\leq \epsilon_1}} \left(\sum\limits_{\substack{ \pi_1\in NC([l_1])\\ \sigma_1\leq \pi_1\leq \epsilon_1}}\mu_{}(\sigma_1,\pi_1)\right)\Phi_{\sigma_1}\bigg(a_1,\cdots ,a_{l_1-1}A_{1}\bigg).\\
	\end{array}
	\end{equation}
	We now fix $\sigma_1\in NC(l_1), \sigma_1\leq \epsilon_1$. 
	We shall express $\Phi_{\sigma_1}\bigg(a_1,\cdots ,a_{l_1-1}A_{1}\bigg)$ according to the definition given by (\ref{def:n-B-map}).
	We need to know how $\Phi_{\sigma_1}$ is decomposed. To this end, 
	suppose that $V$ is the block of $\sigma_1$ which contains 
	the element $l_1$. 
	Denote that 
	$V=\{p_1,p_2,\cdots,p_{k_1}\}$, where $p_{k_1}=l_1$.  
	Set  $W_1=[1,p_1-1]$,  $W_2=[p_1+1,p_2-1],\cdots,W_{k_1}=[p_{k_1-1}+1,p_{k_1}-1]=[p_{k_1-1}+1,l_1-1]$
	($W_i$ will be the empty set if $p_{i-1}+1=p_{i}$), as illustrated in the picture below.  

$\ $
	
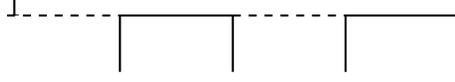
	\begin{center}
  \thicklines
\begin{tikzpicture}[scale=1.5]

\draw [thick, dashed] (0,0)-- (1,0);
\draw[thick] (1,0)--(2,0);
\draw[thick, dashed] (2,0)--(3,0);
\draw[thick] (3,0)--(4,0);
\draw [thick] (1,0)--(1,-0.5);
\draw [thick] (2,0)--(2,-0.5);
\draw [thick] (3,0)--(3,-0.5);
\draw [thick] (4,0)--(4,-0.5);
\put (3, -6) {$W_1$};
\put (20, -6) {$W_2$};
\put (50, -6) {$W_{k_1}$};
\put (14, -12) {$p_1$};
\put (28, -12) {$p_2$};
\put (42, -12) {$p_{k_1-1}$};
\put (56, -12) {$p_{k_1}=l_1$};
\end{tikzpicture}

  \end{center}

	$\ $
	
	Note that $l_1\not\in W_i$ for all $1\leq i \leq k_1$, we have 
	\begin{equation}\label{eq:002}
		 \begin{array}{rcl}
	&&\Phi_{\sigma_1}(a_1,\cdots ,a_{l_1-1}, A_{1})\\
	 &=&\E\left\{\left(\prod\limits_{k=1}^{k_1-1} \left[ \Phi_{\sigma_1|_{W_k}}(a_1,\cdots, a_{l_1-1},A_1)|_{W_k}\right]a_{p_k}\right)\left[\Phi_{\sigma_1|_{W_{k_1}}}(a_1,\cdots, a_{l_1-1},A_1)|_{W_{k_1}}\right]A_1\right\}\\ [12pt] 
	 	 &=&\E\left\{\left(\prod\limits_{k=1}^{k_1-1} \left[ \Phi_{\sigma_1|_{W_k}}(a_1,\cdots, a_{l_1-1})|_{W_k}\right]a_{p_k}\right)
	 	 \left[\Phi_{\sigma_1|_{W_{k_1}}}(a_1,\cdots, a_{l_1-1})|_{W_{k_1}}\right]A_1\right\}\\ [12pt]
	 	 &=&\E\left\{\left(\prod\limits_{k=1}^{k_1-1} \left[ \Phi_{\sigma_1|_{W_k}}(a_1,\cdots, a_{l_1-1})|_{W_k}\right]a_{p_k}\right)
	 	\left[ \Phi_{\sigma_1|_{W_{k_1}}}(a_1,\cdots, a_{l_1-1})|_{W_{k_1}}\right] 
		a_{l_1}\cdots a_m\right\}.
	\end{array}
\end{equation}
  Denote by $A_2=\left(\prod\limits_{k=1}^{k_1-1} \left[ \Phi_{\sigma_1|_{W_k}}(a_1,\cdots, a_{l_1-1})|_{W_k}\right]a_{p_k}\right)
  \left[ \Phi_{\sigma_1|_{W_{k_1}}}(a_1,\cdots, a_{l_1-1})|_{W_{k_1}}\right]a_{l_1}$. 
	Notice that $|\chi^{-1}(\bb)\cap[l_1+1,n-1]|=m-2$ . 
	We now apply the induction formula for the tuple $(A_2, a_{l_1+1}, \cdots, a_n)$,
	recall that $\chi'$ is the restriction of $\chi$ to the interval $[l_1, n]$
	and $\epsilon'$ is the restriction of $\epsilon$ to the interval $[l_1,n]$, we deduce that
	\begin{equation}\label{eq:003}
	\begin{array}{rcl}
	&& \Phi_{\sigma_1}(a_1,\cdots ,a_{l_1-1}, A_{1})=
	\E(A_2 a_{l_1+1}a_{l_1+2}\cdots a_n)\\
	&=&\sum\limits_{\sigma'\in INC(\chi')} \left(\sum\limits_{\substack{ \pi'\in INC(\chi')\\ \sigma'\leq \pi'\leq \epsilon'}}\mu_{INC}(\sigma',\pi')\right)\Phi_{\sigma'} (A_2 ,a_{l_1+1},a_{l_1+2},\cdots ,a_n). 
	\end{array}
	\end{equation}
	
	We now fix $\sigma'\in INC(\chi'), \sigma'\leq \epsilon'$.
	We need to express $\Phi_{\sigma'} (A_2, a_{l_1+1},a_{l_1+2},\cdots, a_n)$
	according to the definition given in (\ref{def:n-B-map}). To this end, 
	suppose that $V'$ is the block of $\sigma'$ which contains the element $l_1$. 
	Suppose that 
	$V'=\{q_1,q_2,\cdots,q_{k_2}\}$, where $q_{1}=l_1$.  Let  
	$W'_1=[l_1,q_2-1]$,  $W'_2=[q_2+1,q_3-1],\cdots,W'_{k_2}=[q_{k_2}+1,n]$ ($W'_j=\emptyset$
	if $q_{j}+1=q_{j+1}$), as shown in the picture below
 	
	$\ $	
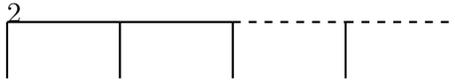
	 	\begin{center}
  \thicklines
\begin{tikzpicture}[scale=1.5]

\draw [thick] (0,0)-- (1,0);
\draw[thick] (1,0)--(2,0);
\draw[thick, dashed] (2,0)--(3,0);
\draw[thick, dashed] (3,0)--(4,0);
\draw [thick] (0,0)--(0,-0.5);
\draw [thick] (1,0)--(1,-0.5);
\draw [thick] (2,0)--(2,-0.5);
\draw [thick] (3,0)--(3,-0.5);
\put (3, -6) {$W_1'$};
\put (20, -6) {$W_2'$};
\put (50, -6) {$W_{k_2}'$};

\put (-5, -12) {$l_1=q_1$};
\put (14, -12) {$q_2$};
\put (28, -12) {$q_{3}$};
\put (42, -12) {$q_{k_2}$};
\end{tikzpicture}

  \end{center}
 
$\ $

	Notice that $l_1\not\in W'_j $ for all $1\leq j\leq q_{k_2}$, we apply
	the induction assumption to the tuple $(A_2, a_{l_1+1},a_{l_1+2},\cdots, a_n)$
	to obtain the following:
	
	\begin{equation} \label{eq:004}
	\begin{array}{rcl}
	&&\Phi_{\sigma'} (A_2, a_{l_1+1},a_{l_1+2},\cdots, a_n)\\ [5pt]
	
		&=&\E\left\{A_2\left[\Phi_{\sigma'|_{W'_1}}(A_2, a_{l_1+1},\cdots, a_{n})|_{W'_1}\right]
	\left(\prod\limits_{k=2}^{k_2} a_{q_k}
	\left[\Phi_{\sigma'|{W'_k}}(A_2, a_{l_1+1},\cdots, a_{n})|_{W'_k}\right]\right)\right\}\\ [12pt]
	
	&=&\E\left\{A_2\left[\Phi_{\sigma'|_{W'_1}}( a_{l_1+1},\cdots, a_{n})|_{W'_1}\right]
	\left(\prod\limits_{k=2}^{k_2} a_{q_k}
	\left[\Phi_{\sigma'|{W'_k}}( a_{l_1+1},\cdots, a_{n})|_{W'_k}\right]\right)\right\}\\ [12pt]
	
	&=&\E \Bigg\{
		 \left(\prod\limits_{k=1}^{k_1-1} \left[ \Phi_{\sigma_1|_{W_k}}(a_1,\cdots, a_{l_1-1})|_{W_k}\right]a_{p_k}\right)
		 \left[ \Phi_{\sigma_1|_{W_{k_1}}}(a_1,\cdots, a_{l_1-1})|_{W_{k_1}}\right]\\ [12pt]
		&&
		a_{l_1}\left[\Phi_{\sigma'|_{W'_1}}(a_{l_1+1},\cdots, a_{n})|_{W'_1}\right]
		\left(\prod\limits_{k=2}^{k_2} a_{q_k}
		\left[\Phi_{\sigma'|{W'_k}}( a_{l_1+1},\cdots, a_{n})|_{W'_k}\right]\right)
	\Bigg\}
	
	\end{array}
	\end{equation}
	
	Recall that $\alpha'_1(\pi):=\left(\alpha_1(\pi), \alpha'(\pi)\right)$
	defined in Proposition \ref{lattice isomorphism}. Let $\sigma=\alpha'^{-1}_1(\sigma_1,\sigma')$.
	We draw the picture below to show the block $V\in \sigma_1$ and the block $V'\in \sigma_2$
which contain $l_1$.
	
	$\ $
	
	\begin{center}
  \thicklines
\begin{tikzpicture}[scale=1.5]

\draw [thick, dashed] (0,0)-- (1,0);
\draw[thick] (1,0)--(2,0);
\draw[thick, dashed] (2,0)--(3,0);
\draw[thick] (3,0)--(4,0);
\draw [thick] (1,0)--(1,-0.5);
\draw [thick] (2,0)--(2,-0.5);
\draw [thick] (3,0)--(3,-0.5);
\draw [thick] (4,0)--(4,-0.5);
\put (3, -6) {$W_1$};
\put (20, -6) {$W_2$};
\put (50, -6) {$W_{k_1}$};
\put (14, -12) {$p_1$};
\put (28, -12) {$p_2$};
\put (42, -12) {$p_{k_1-1}$};
\put (52, -16) {$p_{k_1}=l_1=q_1$};

\draw [thick] (0+4,0)-- (1+4,0);
\draw[thick] (1+4,0)--(2+4,0);
\draw[thick, dashed] (2+4,0)--(3+4,0);
\draw[thick, dashed] (3+4,0)--(4+4,0);
\draw [thick] (0+4,0)--(0+4,-0.5);
\draw [thick] (1+4,0)--(1+4,-0.5);
\draw [thick] (2+4,0)--(2+4,-0.5);
\draw [thick] (3+4,0)--(3+4,-0.5);
\put (65, -6) {$W_1'$};
\put (80, -6) {$W_2'$};
\put (110, -6) {$W_{k_2}'$};

\put (73, -12) {$q_2$};
\put (88, -12) {$q_{3}$};
\put (103, -12) {$q_{k_2}$};
\end{tikzpicture}

  \end{center}

$\ $

	Then, 
	$$ 
	\begin{array}{rcl}
	 	&&\Phi_{\sigma'} (A_2, a_{l_1+1},a_{l_1+2},\cdots, a_n) \\  [5pt]
	&=&\E\Bigg\{ \left( \prod\limits_{k=1}^{k_1-1} \left[\Phi_{\sigma_1|_{W_k}}(a_1,\cdots, a_{l_1-1})|_{W_k}\right]a_{p_k}\right)
	\Phi_{\sigma_1|_{W_k}}(a_1,\cdots, a_{l_1-1})|_{W_k}\\   [5pt]
	&&a_{l_1}\left[\Phi_{\sigma'|_{W'_1}}(a_{l_1+1},\cdots, a_{n}|_{W'_1})\right]
	\left(\prod\limits_{k=2}^{k_2} a_{q_k}\Phi_{\sigma'|_{W'_k}}(a_{l_1+1},\cdots, a_{n}|_{W'_k})\right)\Bigg\}\\  [5pt]
	&=&\Phi_{\sigma} (a_1,\cdots, a_n). \\
	\end{array}
	$$
	Putting $(\ref{eq:001}), (\ref{eq:002}), (\ref{eq:003}), (\ref{eq:004})$ together, we have
	\[
	  \begin{array}{rcl}
	  &&\E(a_1\cdots a_n)\\
	    	&=&\sum\limits_{\substack{\sigma_1\in NC([l_1]) \\ \sigma_1\leq \epsilon_1}} \left(\sum\limits_{\substack{ \pi_1\in NC([l_1])\\ \sigma_1\leq \pi_1\leq \epsilon_1}}\mu_{}(\sigma_1,\pi_1)\right)\Phi_{\sigma_1}\bigg(a_1,\cdots ,a_{l_1-1}A_{1}\bigg).\\
	    	
	  &=&
	  \sum\limits_{\substack{\sigma_1\in NC([l_1]) \\ sigma_1\leq \epsilon_1}} \left(\sum\limits_{\substack{ \pi_1\in NC([l_1])\\ \sigma_1\leq \pi_1\leq \epsilon_1}}\mu_{}(\sigma_1,\pi_1)\right)
	  \sum\limits_{\sigma'\in INC(\chi')} \left(\sum\limits_{\substack{ \pi'\in INC(\chi')\\ \sigma'\leq \pi'\leq \epsilon'}}\mu_{INC}(\sigma',\pi')\right)\Phi_{\sigma'} (A_2 ,a_{l_1+1},a_{l_1+2},\cdots ,a_n)\\
	  &=&
	    \sum\limits_{\substack{\sigma_1\in NC([l_1]) \\ \sigma_1\leq \epsilon_1}} \left(\sum\limits_{\substack{ \pi_1\in NC([l_1])\\ \sigma_1\leq \pi_1\leq \epsilon_1}}\mu_{}(\sigma_1,\pi_1)\right)
	    \sum\limits_{\sigma'\in INC(\chi')} \left(\sum\limits_{\substack{ \pi'\in INC(\chi')\\ \sigma'\leq \pi'\leq \epsilon'}}\mu_{INC}(\sigma',\pi')\right)
	    \Phi_{\alpha'^{-1}_1(\sigma_1,\sigma')} (a_1,\cdots ,a_n)\\  [15pt]
	   &=&\sum\limits_{\sigma\in INC(\chi)} \Big(\sum\limits_{\substack{ \pi\in INC(\chi)\\ \sigma\leq \pi\leq \epsilon}}\mu_{INC}(\sigma,\pi)\Big)\Phi_{\sigma} (a_1,\cdots ,a_n),
	  \end{array}
	\]
	where we used Corollary 3.8 in the last identity and thus
	we obtained our desired equation.

\end{proof}

\section{limit theorems}
Let $A=\big( (a_i)_{i\in \I}, (a_j)_{j\in \J} \big)$ 
be a two faced family of noncommutative random variables in a $\B$-valued probability space
$(\A, \E)$. Let $\omega:\{1,\cdots, n\}\rightarrow \I \bigsqcup\J$
and denote by $\chi_\omega:\{1,\cdots, n\}\rightarrow\{ \f, \bb\}$
the map such that $\chi_\omega(k)=\bb$ if and only if $\omega(k)\in \J$.
\begin{definition}\normalfont
	A two-faced family $A=\big( (a_i)_{i\in \I}, (a_j)_{j\in \J} \big)$
	in a $B$-valuede probability space $(\A,\E)$ is said
	to have a centered free-Boolean limit if, for all $n\neq 2$, 
	\[
	\kappa_{{\chi_\omega}, _{1_n}}\left(
	a_{\omega(1)}b_1, \cdots, a_{\omega(n-1)}b_{n-1}, a_{\omega(n)} \right)=0,
	\]
	for all $\omega:\{1,\cdots,n\}\rightarrow \I\bigsqcup\J$ and $b_1,\cdots, b_{n-1}\in \B$.
	
	The distribution defined by the 
	the two faced family $A$ 
	is called an operator-valued free-Boolean Gaussian distribution
	with covariance $C=(c_{i,j})_{i,j\in \I\bigsqcup\J}$, where $C$ is defined by
	$c_{\omega(1),\omega(2)}(b):=\kappa_{\chi_\omega, 1_2}(a_{\omega(1)}, ba_{\omega(2)})$
	for all $\omega:\{1,2\}\rightarrow \I\bigsqcup \J$ and $b\in \B$.
\end{definition}

\begin{proposition}\normalfont
	Let $A=\big( (a_i)_{i\in \I}, (a_j)_{j\in \J} \big)$ be a two faced family 
	of noncommutative random varialbes in a $\B$-valued probability space $(\A, \E)$. Let $\omega:\{1,2\}\rightarrow\I\bigsqcup\J$.
	Then,
	\[
	\kappa_{{\chi_\omega}, _{1_2}}(a_{\omega(1)}, a_{\omega(2)})
	=\E(a_{\omega(1)} a_{\omega(2)})-\E(a_{\omega(1)})\E( a_{\omega(2)}).
	\]
\end{proposition}

\begin{theorem}\normalfont
	Let $ \Big\{A_m=\big( (a_{m,i})_{i\in \I}, (a_{m,j})_{j\in \J} \big)\Big\}_{m=1}^{\infty}$
	be a free-Boolean sequence of families in a $B$-valued 
	probability space $(\A, \E)$ where $\B$ is a Banach space, such that
	\begin{enumerate}
		\item $\E(a_{m,k})=0$ for all $m\in\mathbb{N}$ and $k\in \I\bigsqcup\J$.
		\item $\sup\limits_{m\in\mathbb{N}}|| \E(a_{m,\omega(1)}b_1 \cdots 
		a_{m,\omega(n-1)}b_{n-1}a_{m,\omega(n)}) ||\leq D_{\omega}<\infty$
		for all $n\in \mathbb{N}$, $\omega:\{1,\cdots, n\}\rightarrow \I\bigsqcup\J$,
		and $b_1,\cdots, b_{n-1}\in \B$.
		\item $\lim\limits_{N\rightarrow \infty} N^{-1} \sum\limits_{1\leq m\leq N}
		\E(a_{m, \omega(1)}b a_{m,\omega(2)})= c_{\omega(1),\omega(2)}(b)\in \B$,
		for all $\omega:\{1,2\}\rightarrow \I\bigsqcup\J$ and $b\in \B$. 
	\end{enumerate}
	Let $S_{N, k}=N^{-1/2}\sum\limits_{1\leq m \leq N}a_{m,k}$ for $k\in\I\bigsqcup\J$
	and $S_N=\big( (S_{N,i})_{i\in \I}, (S_{N, j})_{j\in \J} \big)$.
	Denote by $\gamma_C$ the free-Boolean limit distribution in Definition ***,
	with $C=(c_{i,j})_{i,j\in \I\bigsqcup\J}$. 
	We have
	\[
	\lim_{N\rightarrow \infty} \mu_{S_N}(P)=\gamma_C(P),
	\]
	for all $P\in \C\langle a_k| k\in \I\bigsqcup\J \rangle$.
\end{theorem}
\begin{proof}
	Since the joint distributions are determined by free-Boolean cumulants uniquely, 
	it is enough to show that
	\[
	\lim_{N\rightarrow \infty}\kappa_{{\chi_\omega}, _{1_n}}\big(
	S_{N,\omega(1)}b_1, \cdots, S_{N,\omega(n-1)}b_{n-1}, S_{N, \omega(n)} \big)
	=\kappa_{{\chi_\omega}, _{1_n}}\big( S_{\omega(1)}b_1,\cdots, S_{\omega(n-1)}b_{n-1},
	S_{\omega(n)} \big),
	\]
	where the two faced family $S=\big( (S_i)_{i\in \I}, (S_j)_{j\in \J} \big)$
	has a centered $\B$-valued free-Boolean Gaussian dsitribution with covariance matrix 
	$C$, for all $n\in \mathbb{N}$,
	$\omega:\{1,\cdots,n\}\rightarrow \I\bigsqcup\J$,
	and $b_1,\cdots, b_{n-1}\in\B$. 
	
	By the additivity property of free-Boolean cumulants, we have
	\begin{align*}
	&\kappa_{{\chi_\omega}, _{1_n}}\big( S_{N,\omega(1)}b_1,\cdots, S_{N,\omega(n-1)}b_{n-1}, 
	S_{N,\omega(n)} \big)\\
	=&\frac{1}{N^{n/2}}\sum_{1\leq m\leq N}\kappa_{\chi_\omega, _{1_n}}\big(
	a_{m, \omega(1)}b_1, \cdots, a_{m,\omega(n-1)}b_{n-1}, a_{m,\omega(n)} \big).
	\end{align*}
	Since the free-Boolean cumulant are universal polynomial of mixed moments, we deduce from 
	assumption $(2)$ that
	\[
	\sup_{m\in \mathbb{N}} || \kappa_{\chi_\omega, _{1_n}}\big(
	a_{m, \omega(1)}b_1, \cdots, a_{m,\omega(n-1)}b_{n-1}, a_{m,\omega(n)} \big)|| <\infty
	\]
	and hence
	\[
	\lim_{N\rightarrow\infty }
	\kappa_{\chi_\omega, _{1_n}}\big( S_{N,\omega(1)}b_1,\cdots, S_{N,\omega(n-1)}b_{n-1}, 
	S_{N,\omega(n)} \big)=0
	\]
	for $n\geq 3$. 
	
	As $\kappa_{\chi_\omega, _{1_1}}(a_{m,\omega(1)})=\E(a_{m,\omega(1)})=0$ for all $m\in\mathbb{N}$ and $\omega:\{1\}
	\rightarrow \I\bigsqcup\J$,
	we have $\kappa_{\chi_\omega, _{1_1}} (S_{N, \omega(1)})=0$ for all $N\in \mathbb{N}$ and $\omega:\{1\}\rightarrow \I\bigsqcup\J$. 
	Finally, by assumption $(3)$, and Proposition 5.2, we have
	\begin{align*}
	&\kappa_{\chi_\omega, _{1_2}}(S_{N,\omega(1)}b, S_{N,\omega(2)})\\
	=&\frac{1}{N}\sum_{1\leq m \leq N}\kappa_{\chi_\omega, _{1_2}}(a_{n,\omega(1)}b, a_{n,\omega(2)})\\
	=&\frac{1}{N}\sum_{1\leq m\leq N} \E(a_{n,\omega(1)}ba_{n,\omega(2)})\rightarrow C_{\omega(1), \omega(2)}(b),  
	\end{align*}
as $N\rightarrow\infty$,	for all $\omega:\{1,2\}\rightarrow \I\bigsqcup\J$ and $b\in \B$. This finishes the proof. 
\end{proof}

\section{Moment-conditions for free-Boolean independence}
Let $\{(\X_i, \mathring{\X}_i, p_i)\}_{i \in \I}$ be $\B$-$\B$-bimodules with specified projectin, 
and $\{\X, \mrx, p \}$ be the reduced free product with amalgamation over $\B$.
For each $i\in\I$, denote $\A_{i,\f}=\lambda_i(\LL(\X_i))$ and
$\A_{i,\bb}=\beta_i(\LL(\X_i))=P_i\lambda_i(\LL(X_i))P_i$, where $P_i$ is the projection onto 
the  subspace $\B\oplus \mrx_i$. 
We also denote by $\A_i$ the algebra generated by $\A_{i,\f}\cup\A_{i,\bb}$.

Given a family $\{C_i, D_i\}_{i\in \I}$ of free-Boolean pair of $B$-faces
in a $B$-valued probability space $(\A, \E)$, to study the mixed
moments of the family, one can identify 
$C_i$ with $\lambda_i(\gamma_i(C_i))$ and 
identify $D_i$ with $\beta_i(\delta_i(D_i))$
following Definition \ref{free-BooleanDef}. 
In this way, we regard $C_i$ as a subalgebra
of $\A_{i,\f}$ and $D_i$ as a subalgebra of $\A_{i,\bb}$
throughout this section. 

\begin{definition}\normalfont
	Given a set $S_i\subset \A_{i,\f}\cup \A_{i,\bb}$
	and $a_1,\cdots, a_m\in S_i$
	their
	product $A=a_1\cdots a_m$ is called a \emph{simple product} 
	of elements from $S_i$.
	It is called a \emph{Boolean product} of elements from $S_i$ if $a_k\in S_i\cap \A_{i,\bb}$ 
	for some $1\leq k\leq m$. 
\end{definition}
If $A=a_1\cdots a_m$ is a simple product, but not a Boolean product, then each 
$a_i \in S_i\cap \A_{i,\f}$. 
A Boolean product of elements from $\A_{i,\f}\cup \A_{i,\bb}$ 
has a very simple form. 
\begin{proposition}\normalfont
	\label{Boolean product I}
Let $A\in \A_i$ ($i\in \I$) be a Boolean product
of elements from $\A_{i,\f}\cup \A_{i,\bb}$. Then $A\in \A_{i,\bb}$,
in particular, $A(\X)\subset\X_i$.
\end{proposition}
\begin{proof}
Write $A=a_1\cdots a_m$ where $a_k\in \A_{i,\f}\cup\A_{i,\bb}$ ($1\leq k\leq m$).
For an element $a_k\in \A_{i,\bb}$, it can be written as
$a_k=P_ia_k'P_i$, where $a_k'\in \A_{k,\f}$ by definition.
Hence $A=P_i b_1\cdots b_m P_i$, where 
\[  
b_k= \left\{
\begin{array}{ll}
 a_k', \quad&\text{if}\quad a_k\in \A_{i,\bb}\\
 a_k, \quad&\text{if}\quad a_k\in \A_{i,\f}
\end{array} 
\right. \]
by Proposition \ref{simple Boolean}. 
The assertion follows. 
\end{proof}

To state our result in general, from now on, we will 
let $S_i\subset \A_{i,\f}\cup\A_{i,\bb}$ for each $i\in \I$.
Typically, $S_i=C_i\cup D_i$ or $S_i=\A_{i,\f}\cup \A_{i,\bb}$. 
\begin{lemma}\label{BFproduct}\normalfont
	Let $A_1\in \A_i$ and $A_2\in \A_j$ be two simple product of elements from $S_i$ and $S_j$ respectively
	and $i\neq j$. If $A_1$ is a Boolean product, then 
	\[
	A_1A_2 1_{\B}=A_1 \E_{\LL(\X)}(A_2). 
	\]
\end{lemma}
\begin{proof}
		Write $A_2 (1_\B)=\E_{\LL(\X)}(A_2)\oplus \mathring{A}_2$, where
	$\mathring{A}_2\in \mrx_j$. As $A_1$ is a Boolean product, 
	at least one of the factor is from $\A_{i,\bb}$, 
	we thus can write 
	$A_1=a_1ba_2$, where $a_2$ is a simple product of elements from $\A_{i,\f}$
	and $b\in \A_{i,\bb}$. We can express 
	$a_2=\lambda_i(T)$, where $T\in \LL_{\B}(\X_i)$.
	
	Observe that
	\begin{align*}
	\lambda_i(T)(\mathring{A}_2)&=V_i (T\otimes I_{\X(i)})V_i^{-1}\mathring{A}_2\\
	&=V_i(T\otimes I_{\X(i)})(1_{\B}\otimes \mathring{A}_2)\\
	&=V_i\big[ \E_{\LL(X_i)}(T)\otimes \mathring{A}_2 + (T-\E_{\LL(\X_i)}(T))\otimes \mathring{A}_2 \big].
	\end{align*}
	Hence, $\lambda_i(T)\mathring{A}_2 \in \mrx_j \oplus (\mrx_i\otimes \mrx_j)$. 
	As $b=P_i b P_i$, where $P_i$ is the projection onto $\B\oplus \mrx_i$, we deduce
	that $b\mathring{A}_2=0$. We then have
	\begin{align*}
	A_1 A_2 1_\B&=(a_1 b a_2) \left( \E_{\LL(\X)}(A_2)+\mathring{A}_2\right)\\
	&=(a_1 b a_2) ( \E_{\LL(\X)}(A_2))\\
	&=A_1 ( \E_{\LL(\X)}(A_2)).
	\end{align*}
	This finishes the proof. 
\end{proof}

An application of the preceeding lemman and 
the bimodule property of expection $\E_{\LL(\X)}$
implies the following result.
\begin{corollary}\label{Boolean Moments}\normalfont
	For $1\leq i\leq m$, let $A_i\in \A_{k(i)}$ be Boolean product 
	of elements from $\A_{k(i),\f}\cup \A_{k(i),\bb}$ and $k(1)\neq k(2)\neq \cdots \neq k(m)$.
	Then, we have
	\[
	\E_{\LL(\X)}(A_1\cdots A_m)=\E_{\LL(\X)}(A_1)\cdots \E_{\LL(\X)}(A_m). 
	\]
\end{corollary}

\begin{lemma}{\label{small space}}
	Let $B\in \A_{i}$ be a Boolean product of elements from $S_i\subset \A_{i,\f}\cup\A_{i,\bb}$,
	then $B(\X)\subset \X_i$. 
\end{lemma}
\begin{proof}
	Observe that $B$ can be written as $B=a_1 b a_2$, where $b\in \A_{i,\bb}$ and $a_1, a_2$ 
	are simple products of elements from $\A_{i,\f}\cup\A_{i,\bb}$. 
	Since $b=P_i b P_i$, where $P_i$ is the projection onto $\X_i=\B\oplus \mrx_i$,
	we deduce that $ba_2(\X)\subset \mrx_i$. Hence the assertion follows. 
\end{proof}

\begin{lemma}\label{zero product} \normalfont
	Let	$A_i\in \A_{k(i)} (1\leq i\leq m)$ be 
	simple products of elements from $S_i\subset \A_{k(i),\f}\cup\A_{k(i),\bb}$. If  the following conditions hold:
	\begin{enumerate}
		\item There exist $1< l_1< l_2< m$ such that $A_{l_1}, A_{l_1+1},\cdots, 
		A_{l_2}$ are not Boolean products.
		\item  $A_{l_1-1}$ and $A_{l_2+1}$ are Boolean products.
		\item  $k(1)\neq k(2)\neq\cdots\neq k(m)$.
		\item $\E_{\LL(\X)}(A_{l_1})=\cdots=\E_{\LL(\X)}(A_{l_2})=0$.
	\end{enumerate}
	Then, the product of operators $A_1\cdots A_m=0$.
\end{lemma}
\begin{proof}
		Since $A_{l_2+1}$ is a Boolean product,
	it follows from Lemma \ref{small space} that 
	\[
	A_{l_2+1}\cdots A_m (\X) \subset B\oplus \mrx_{k(l_2+1)}.
	\]	
	As $\E_{\LL(X)}(A_{l_1})=\cdots=\E_{\LL(X)}(A_{l_2})=0$
	and each $A_j$ ($l_1\leq j\leq l_2$) is a simple 
	product of elements from $\A_{k(j),\f}$, in this case, we then have
	\begin{align*}
	A_{l_2}(A_{l_2+1}\cdots A_m)(\X)\subset &
	\mrx_{k(l_2)} \otimes \big(B\oplus \mrx_{k(l_2+1)} \big)\\
	\cong & \mrx_{k(l_2)} \oplus 
	\big(\mrx_{k(l_2)}\otimes \mrx_{k(l_2+1)} \big).
	\end{align*}
	By induction, we then have
	\begin{align}\label{eq:009}
	\begin{split}
	(A_{l_1}\cdots A_{l_2})&(A_{l_2+1}\cdots A_m)(\X)
	\\
	\subset & \big( \mrx_{k(l_1)}\otimes\cdots\otimes \mrx_{k(l_2)} \big)\oplus 
	\big(\mrx_{k(l_1)}\otimes\cdots\otimes \mrx_{k(l_2)}\otimes \mrx_{k(l_2+1)} \big).
	\end{split}
	\end{align}
	
	The operator $A_{l_1-1}$ is a Boolean product.
	We may write
	$A_{l_1-1}=a_1 b a_2$ where $a_2$ is a simple product of elements 
	from $\A_{k(l_1-1),\f}$ and $b\in \A_{k(l_1-1), \bb}$. If follows that
	\begin{align*}
	a_2 (A_{l_1}\cdots  A_m) (\X) \subset 
	& \X_{k(l_1-1)}\otimes \big[  \big( \mrx_{k(l_1)}\otimes\cdots\otimes \mrx_{k(l_2)} \big)\oplus 
	\big(\mrx_{k(l_1)}\otimes\cdots\otimes \mrx_{k(l_2)}\otimes \mrx_{k(l_2+1)} \big) \big]\\
	\cong &\big( \mrx_{k(l_1)}\otimes\cdots\otimes \mrx_{k(l_2)} \big)\oplus 
	\big(\mrx_{k(l_1)}\otimes\cdots\otimes \mrx_{k(l_2)}\otimes \mrx_{k(l_2+1)} \big)\\
	\oplus & (\mrx_{k(l_1-1)} \otimes \mrx_{k(l_1)}\otimes\cdots\otimes \mrx_{k(l_2)}) 
	\oplus (\mrx_{k(l_1-1)} \otimes\mrx_{k(l_1)}\otimes\cdots\otimes \mrx_{k(l_2)}\otimes \mrx_{k(l_2+1)}).
	\end{align*}
	
	As $b=P_{k(l_1-1)} bP_{k(l_1-1)}$, where $P_{k(l_1-1)}$ is the projection onto $\B\oplus \mrx_{k(l_1-1)}$.
	We then have
	\[
	b a_2 (A_{l_1}\cdots A_m) ( x) =0\quad \text{for all $x\in \X$}.
	\]
	which implies that  $A_1\cdots A_m=0$. This finishes the proof.
\end{proof}

\begin{remark}\normalfont
When $S_i=\A_{i,\f}\cup \A_{i,\bb}$, note that
$\A_{i,\bb}=P_i \A_{i,\f}P_i$ and Proposition \ref{simple Boolean},
 the Boolean product $A_{l_1-1}$
can be written as 
$A_{l_1-1}=P_{k(l_1-1)} a P_{k(l_1-1)}$, where 
	$a\in \A_{k(l_1-1), \f}$ following Proposition \ref{Boolean product I}. 
	The proof for Proposition \ref{zero product}
	and the proof for Proposition \ref{Free-boolean moments} can be simplified. 
\end{remark}

\begin{proposition}\label{Free-boolean moments} \normalfont
	Let	$A_i\in \A_{k(i)} (1\leq i\leq m)$ be 
	simple products of elements from $S_i\subset \A_{k(i),\f}\cup\A_{k(i),\bb}$. If  the following conditions hold:
	\begin{enumerate}
		\item There exist $1\leq l_1< l_2\leq m$ such that $A_{l_1}, A_{l_1+1},\cdots, 
		A_{l_2}$ are not Boolean products.
		\item Either $l_1=1$ or $A_{l_1-1}$ is a Boolean product.
		\item Either $l_2=m$ or $A_{l_2+1}$ is a Boolean product.
		\item  $k(1)\neq k(2)\neq\cdots\neq k(m)$.
		\item $\E_{\LL(\X)}(A_{l_1})=\E_{\LL(X)}(A_{l_1+1})=\cdots=\E_{\LL(\X)}(A_{l_2})=0$.
	\end{enumerate}
	Then, we have $\E_{\LL(\X)}(A_1\cdots A_m)=0$.
\end{proposition}
\begin{proof} 
		If $1<l_1<l_2<m$, the asseration follows immediately from Lemma \ref{zero product}.	
	
	If $l_2<m, l_1=1$, from the proof of Lemma \ref{zero product}, we see that
	\begin{align*}
	(A_1\cdots A_m)1_\B=&(A_{l_1}\cdots A_{l_2})(A_{l_2+1}\cdots A_m)1_\B 
	\\
	\in & \big( \mrx_{k(l_1)}\otimes\cdots\otimes \mrx_{k(l_2)} \big)\oplus 
	\big(\mrx_{k(l_1)}\otimes\cdots\otimes \mrx_{k(l_2)}\otimes \mrx_{k(l_2+1)} \big).
	\end{align*}
	Hence the asseration holds in this case as well.
	
	If $l_2=m$, then 
	the assumptions $(4)$ and $(5)$ imply that
	\begin{align}\label{eq:011}
	(A_{l_1}\cdots A_{l_2})1_\B 
	\in \big( \mrx_{k(l_1)}\otimes\cdots\otimes \mrx_{k(l_2)} \big).
	\end{align}
	
	We now have two cases:
	(i) If further $l_1=1$, it is clear that
	$
	\E_{\LL(\X)}(A_1\cdots A_m)=0,
	$
	thanks to (\ref{eq:011}).
	(ii) If $l_1>1$, then $A_{l_1-1}$ is a Boolean product. We may write
	$A_{l_1-1}=a_1 b a_2$ where $a_2$ is a simple product of elements 
	from $\A_{k(l_1-1),\f}$ and $b\in \A_{k(l_1-1), \bb}$. If follows that
	\begin{align*}
	a_2 (A_{l_1}\cdots  A_m) 1_\B \in 
	& \X_{k(l_1-1)}\otimes  \big( \mrx_{k(l_1)}\otimes\cdots\otimes \mrx_{k(l_2)} \big)\\
	\cong &\big( \mrx_{k(l_1)}\otimes\cdots\otimes \mrx_{k(l_2)} \big)
	\oplus (\mrx_{k(l_1-1)} \otimes \mrx_{k(l_1)}\otimes\cdots\otimes \mrx_{k(l_2)}).
	\end{align*}
	As $b=P_{k(l_1-1)} bP_{k(l_1-1)}$, where $P_{k(l_1-1)}$ is the projection onto $\B\oplus \mrx_{k(l_1-1)}$.
	We have
	\[
	ba_2 (A_{l_1}\cdots A_m) 1_\B =0, 
	\]
	which implies that  $\E_{\LL(\X)}(A_1\cdots A_m)=0$ as well.
	This finishes the proof.
\end{proof}


\begin{defprop}\normalfont \label{defprop}
	Proposition \ref{Boolean Moments}, Lemma \ref{zero product} and Proposition \ref{Free-boolean moments} provide us an algorithm for computing mixed moments of free-Boolean independent pairs of random variables 
	and a canonical way to simplify an arbitrary element as follows.
	Denote by the algebra $\A$ generated by $\bigcup_{i\in \I}\{C_i\cup D_i\}$,
	where $\{C_i, D_i\}_{i\in \I}$ is a family of free-Boolean pair of $B$-faces
	in the $B$-valued probability space $(\LL(X), \E_{\LL(\X)})$.
	That is, $C_i\subset \A_{i,\f}, D_i\subset \A_{i,\bb}$ 
	are subalgebras. 
	Then
	\[
	\A = \text{span}\left\{  A_1\cdots A_m\ \middle\vert \begin{array}{l}
	\text{each $A_i$ is a simple product of element from $\A_{\omega(i)}$, }  \\
	\text{and}\quad \omega(1)\neq\omega(2)\neq\cdots\neq\omega(m)
	\end{array}\right\}.
	\]
	Let $X= A_1\cdots A_m$, where each $A_i$ is a simple product from $\A_{\omega(i)}$ and
	$\omega(1)\neq\omega(2)\neq\cdots\neq\omega(m)$. Whenever some $A_k$ ($1\leq k\leq m$)
	is a simple product of element from $\A_{\omega(k),\f}$, we replace it by
	\[
	A_k=(A_k-\E(A_k)1_A)+\E(A_k)1_A.
	\]
   Set $S_i=C_i\cup D_i$ ($i\in \I$). In viewing Lemma \ref{zero product}, Proposition \ref{Boolean product I} and the fact that
	$\A_{i, \f}$ is an algebra, the operator $A_1\cdots A_m$ can be expressed as the sum
	of following types of products:
	\begin{enumerate}
		\item $Z_{(0)}=b 1_\A$, where $b\in \B$. 
	
		\item $Z_{f}=F_1\cdots F_k$, where $k\in \mathbb{N}$, each $F_i\in C_{\omega(i)}\subset\A_{\omega(i),\f}$, $\E(F_i)$ and $\omega(1)\neq\cdots\neq\omega(k)$.
	
		\item $Z_{b}=B_1\cdots B_k$, where $k\in \mathbb{N}$, each 

		each $B_i$ is a Boolean product of elements from $S_{\omega(i)}$
	
		and $\omega(1)\neq\cdots\neq\omega(k)$.
	
		\item $Z_{fb}=F_1\cdots F_{k_1}B_1\cdots B_{k_2}$, where
		$k_1,k_2\in\mathbb{N}$, each $F_i\in C_{\omega(i)}\subset\A_{\omega(i),\f}$ 
		such that $\E(F_i)=0$ for $1\leq i\leq k_1$,
		each
		$B_j$ is a Boolean product of elements from  $S_{\omega(k_1+j)}$ 
		for $1\leq j\leq k_2$,
		and $\omega(1)\neq\cdots\neq \omega(k_1+k_2)$.
	
		\item $Z_{bf}=B_1\cdots B_{k_1}F_1\cdots F_{k_2}$, where 
		$k_1, k_2\in \mathbb{N}$, each $F_i\in C_{\omega(i)}\subset\A_{\omega(i),\f}$ 
		such that $\E(F_i)=0$ for $1\leq i\leq k_2$,
	  each $B_j$ is a Boolean product of elements from $S_{\omega(k_1+j)}$
		for $1\leq j\leq k_1$,
		and $\omega(1)\neq\cdots\neq \omega(k_1+k_2)$.
	
		\item $Z_{fbf}=F_1\cdots F_{k_1}B_1\cdots B_{k_2}F_{k_1+1}\cdots F_{k_1+k_3}$, 
		where $k_1, k_2, k_3\in\mathbb{N}$, each $F_i\in C_{\omega(i)}\subset\A_{\omega(i),\f}$ 
		for $1\leq i\leq k_1$, each 
		$F_{k_1+j}\in C_{\omega(k_1+k_2+j)}\subset\A_{\omega(k_1+k_2+j),\f}$ for $1\leq j\leq k_3$ such that $\E(F_i)=0$
		for $1\leq i \leq k_1+k_3$,
		 $B_j$ is a Boolean product of elements from $S_{\omega(k_1+j)}$ 
		for $1\leq j\leq k_1$,
		and $\omega(1)\neq \cdots\neq \omega(k_1+k_2+k_3)$. 
	\end{enumerate}
	Furthermore, $\E(Z_f)=0$ by the definition of free independence,
	$\E(Z_{bf})=\E(Z_{fb})=\E(Z_{fbf})=0$ by Proposition \ref{Free-boolean moments}.
\end{defprop}

\begin{corollary}\label{cor:6.8}\normalfont
	Given an operator $Z_{fb}$
	and an operator $Z_{bf}$ of the form in Definition-Proposition \ref{defprop} $(4)$
	and $(5)$ respectively, for any $A, B\in \A$, we have 
	\[
	\E(Z_{fb}A)=\E(BZ_{bf})=0.
	\]
\end{corollary}
\begin{proof}
	By linearity, it is enough to consider the case when $A$ is the product of simple products.
	Let $Z_{fb}=F_1\cdots F_{k_1}B_1\cdots B_{k_2}$ by definition. The Boolean product
	$B_{k_2}$ is concatenated with some factors in $A$ to be a Boolean product.
	It then follows from Proposition \ref{Free-boolean moments} that $\E(Z_{fb}A)=0$. 
	The other case can be proved in the same way. 
\end{proof}

\begin{corollary}\label{cor:6.9}\normalfont
	Given an operator $Z_{fbf}$ of the form in Definition-Proposition \ref{defprop} $(6)$,
	for any $A, B\in \A$, we have $\E(Z_{fbf}A)=\E(BZ_{fbf})=0$.
\end{corollary}
\begin{proof}
	It is enough to consider the case when $A$ and $B$ are product 
	of simple products. Noticing Proposition \ref{zero product} and applying the simplification method
	described in Definition-Proposition \ref{defprop},
	it is easy to see that $Z_{fbf}A$ can be written
	as the summation of element of the types $(3)$ and $(5)$, whose
	expectations are zero. Hence $\E(Z_{fbf}A)=0$.
	Similary, $BZ_{fbf}$ can be written
	as the summation of elemtns of the types $(4)$ and $(5)$. 
	Hence $\E(BZ_{fbf})=0$.
\end{proof}

\begin{lemma}\label{lemma:6.10}\normalfont
	Let $B_1\in \A_i$ and $B_2\in \A_j$ be two Boolean products and $i\neq j$.
	Then for any $A\in \A$, we have
	\[
	\E(B_1B_2A)=\E(B_1)\E(B_2A).
	\]
\end{lemma}
\begin{proof}
	Write
	$B_2 A (1_\B)= \E(B_2A)+(B_2A 1_\B-\E(B_2A)) \in \B\oplus \mrx_j$. 
	We can express the Boolean product $B_1$ as $B_1=P_i a P_i$ following
	Proposition \ref{Boolean product I}, and therefore $B_1 (\mrx_j)=0$,
	which shows that
	\[
	\E(B_1B_2A)=\E(B_1 (\E(B_2A)))=\E(B_1)\E(B_2A).
	\]
\end{proof}

The mixed moments $\E(A_1\cdots A_n)$ can be expressed as a universal polynomial of moments 
of elements in individual algebras $\A_i$. 
Thus we have the following  equivalent definition for free-Boolean independence under Moments conditions.
\begin{proposition}\normalfont
	 Let $\{(C_i, D_i)\}_{i\in \I}$
	  be a family of pairs of algebras in a $B$-valued  probability space $(\A, \E)$. 
	Set $S_i=C_i\cup D_i$. 
	The family  $\{(C_i, D_i)\}_{i\in \I}$
	is free-Boolean
	independent if and only if 
	\begin{enumerate}
		\item whenever $B_1,\cdots, B_m$ are operators such 
		that:
		\begin{itemize}
			\item for each $1\leq k\leq m$, $B_k$ is a simple product of elements
			from $S_{\omega(k)}$, at least one of them is from $D_{\omega(k)}$;
			\item $\omega(1)\neq\cdots \neq \omega(m)$.
		\end{itemize}	
	then
		$$\E(B_1\cdots B_m)=\E(B_1)\cdots \E(B_m).$$
		\item whenever 
		$A_1,\cdots, A_m$ are operators such that:
		\begin{itemize}
			\item For each $1\leq k\leq m$, $A_k$ is a product of elements from $S_{\omega(k)}$.
			\item There exist $1\leq l_1< l_2\leq m$ such that 
			$A_{k}$ is a product of elements from $C_{\omega(k)}$ for all $l_1\leq k\leq k_2$. 
			\item Either $l_1=1$ or $A_{l_1-1}$
			is a product of element from $S_{\omega(l_1-1)}$, at least one of them is 
			in $D_{\omega(l_1-1)}$.
			\item Either $l_2=m$ or $A_{l_1+1}$
			is a product of element from $S_{\omega(l_1+1)}$, at least one of them is 
			in $D_{\omega(l_2+1)}$. 
			\item  $\omega(1)\neq \cdots\neq \omega(m)$.
			\item $\E(A_{l_1})=\E(A_{l_1+1})=\cdots=\E(A_{l_2})=0$.
		\end{itemize}
		Then, we have $\E(A_1\cdots A_m)=0$.
	\end{enumerate}
\end{proposition}

\section{Positivity of the amalgamated free-Boolean product}
In this section, we deal with $B$-functionals with positivity property.  For the notion of positivity, we need a $*$-structure on our algebras.  We assume the algebra $B$ has a nice positivity structure, i.e. we demand it to be a unital $C^*$-algebra. For $*$-algebra $\A$, no such restriction is required.

\begin{definition} \normalfont
Let $\A$ be a unital $*$-algebra, element $a\in \A$ is said to be positive if there exists a $b\in \A$ such that $a=bb^*$. A $B$-linear functional $\E$ is said to be positive if $\E(a)$ is positive for all positive element $a\in \A$. A $B$-linear functional $\E$ is said to be unital if $\E(1_\A)=1_B$.
\end{definition}

In the rest of this section, we always assume that $\A$ is a unital $*-$algebra and $\E$ is unital.
Let  $\{(C_i,D_i)\}_{i\in \I}$ ia a family of  $B$-faces  in a $B$-probability space $(\A, \E)$ , which generates $\A$. Suppose that the family $\{(C_i,D_i)\}_{i\in \I}$ is free-Boolean independent with amalgamation over $B$ and $C_i, D_i$ are $*$-subalgebras of $\A$ for all $i\in \I$. For each $i$, let $ A_i$ be the unital $*$-algebras generated by $C_i, D_i$. Let $\E_i$ be the restriction of $\E$ to  $\A_i$.  Then  $(\A_i,\E_i)$ is a $B$-valued probability space. We assume that  $E_i$ is unital and positive and unital for all $i$. 

For convenience, we also introduce the following definition. Then,  
results in Proposition-Definition \ref{defprop} hold
in this abstract framework. 
\begin{definition}\normalfont
	Given a set $S_i\subset C_i\cup D_i$
	and $a_1,\cdots, a_m\in S_i$
	their
	product $A=a_1\cdots a_m$ is called a \emph{simple product} 
	of elements from $S_i$.
	It is called a \emph{Boolean product} of elements from $S_i$ if $a_k\in S_i\cap D_i$
	for some $1\leq k\leq m$. 
\end{definition}

Recall that in the Section(moments-condition),  the algebra $\A$ is the linear span of simple products of type  $Z_0,Z_f, Z_b,Z_{bf},Z_{fs}$ and $Z_{fbf}$, as shown in Proposition-Definition \ref{defprop}. Given $Z\in \A$, then $Z$ can be written as
$$Z=Z_0+\sum\limits_{i_1}Z^{(i_1)}_f+\sum\limits_{i_2}Z^{(i_2)}_b+\sum\limits_{i_3}Z^{(i_3)}_{bf}+\sum\limits_{i_4}Z^{(i_4)}_{fb}+\sum\limits_{i_5}Z^{(i_5)}_{fbf}.$$
We will show that $\E[ZZ^*]$ is positive.
\begin{remark} \normalfont
	$(Z^{(i_3)}_{bf})^*$ is a $Z_{fb}$ type element and  $(Z^{(i_4)}_{fb})^*$ is a $Z_{bf}$ type element.
\end{remark}

We first note that, by Corollary \ref{cor:6.8} and Corollary \ref{cor:6.9},  we have that 
$$\E[ZZ^*]=\E[Z_1Z_1^*],$$
where $Z_1=Z_0+\sum\limits_{i_1}Z^{(i_1)}_f+\sum\limits_{i_2}Z^{(i_2)}_b+\sum\limits_{i_3}Z^{(i_3)}_{bf}$.

To simplify the notation, we introduce the following notations:
\begin{itemize}
			\item For operator of the form 
			$Z_{b}=B_1\cdots B_k$, where $k\in \mathbb{N}$, 
	each $B_i$ is a Boolean product of elemnts from $S_{\omega(i)}$
	and $\omega(1)\neq\cdots\neq\omega(k)$.
	Set 
	\[
	  \begin{array}{crl}
\Psi(Z_b)&=\E(B_1\cdots B_{k-1})B_k, \quad \text{when}\quad k\geq 2,\\
\Psi^*(Z_b)&=B_1\E(B_2\cdots B_{k1}), \quad \text{when}\quad k\geq 2,
	  \end{array}
	\]
	and $\Psi(Z_b)=\Psi^*(Z_b)=Z_b$ when $k=1$. 
	
		\item For operators of the form $Z_{bf}=B_1\cdots B_{k_1}F_1\cdots F_{k_2}$, where 
	$k_1, k_2\in \mathbb{N}$, each $F_i\in C_{\omega(i)}\subset\A_{\omega(i),\f}$ 
	such that $\E(F_i)=0$ for $1\leq i\leq k_2$,
	each $B_j$ is a Boolean product of elements from $S_{\omega(k_1+j)}$
	for $1\leq j\leq k_1$,
	and $\omega(1)\neq\cdots\neq \omega(k_1+k_2)$.
	Set
	\[
	   \Psi(Z_{bf})=\E(B_1\cdots B_{k_1-1})\big(B_{k_1}F_1\cdots F_{k_2}\big), 
	   	\quad \text{when}\quad k_1\geq 2,
	\]
	and $\Psi(Z_{bf})=Z_{bf}=B_1\big(F_1\cdots F_{k_2}\big)$ when $k_1=1$. 
	
	\item For operator of the form $Z_{fb}=F_1\cdots F_{k_1}B_1\cdots B_{k_2}$, where
	$k_1,k_2\in\mathbb{N}$, each $F_i\in C_{\omega(i)}\subset\A_{\omega(i),\f}$ 
	such that $\E(F_i)=0$ for $1\leq i\leq k_1$,
	each
	$B_j$ is a Boolean product of elemnts from  $S_{\omega(k_1+j)}$ 
	for $1\leq j\leq k_2$,
	and $\omega(1)\neq\cdots\neq \omega(k_1+k_2)$.
	Set 
	\[
	   \Psi^*(Z_{fb})=\big(F_1\cdots F_{k_1}B_1 \big)\E(B_2\cdots B_{k_2}),
	   \quad \text{when}\quad k_2\geq 2,
	\]
	and $\Psi^*(Z_fb)=Z_{fb}$ when $k_2=1$.

\end{itemize}

Note that $Z_{bf}^*$ is of the same type as $Z_{fb}$ following 
Proposition-Defintion \ref{defprop} $(4), (5)$.
It is easy to check that the following relation holds:
\[
   \Psi(Z_{bf})^*=\Psi^*(Z_{bf}^*).
\]

\begin{lemma} \label{lemma:7.5}\normalfont
	Let  $Z_{bf}=B_1\cdots B_{k_1}F_1\cdots F_{k_2}$, where each $F_i\in\A_{\omega(i),\f}$ 
	such that $\E(F_i)=0$ for $1\leq i\leq k_2$,
	each $B_j$ is a Boolean product from $\A_{\omega(k_1+j)}$ for $1\leq j\leq k_1$,
	and $\omega(1)\neq\cdots\neq \omega(k_1+k_2)$. Then,
	$\E[Z_{bf}Z']=\E[\Psi(Z_{bf})Z']$
	for all simple products $Z'$.
\end{lemma}
\begin{proof}
It follows by applying Lemma \ref{lemma:6.10} inductively. 
\end{proof}

\begin{lemma}\label{lemma:7.6}\normalfont
	 Let  $Z_{fb}=F_1\cdots F_{k_1}B_1\cdots B_{k_2}$, where each $F_i\in\A_{\omega(i),\f}$ 
	such that $\E(F_i)=0$ for $1\leq i\leq k_1$, $k_1\geq 0$, $k_2\geq 1$ 
	each $B_j$ is a Boolean product from $\A_{\omega(k_1+j)}$ for $k_1<j\leq k_1+k_2$,
	and $\omega(1)\leq \omega(k_1+k_2)$. 
	Then,
	$\E[Z'Z_{fb}]=\E[Z'\Psi^*(Z_{fb})]$
	for all simple products $Z'$. 
\end{lemma}
\begin{proof}
It follows by applying Lemma \ref{BFproduct} inductively. 
\end{proof}

\begin{lemma}\label{lemma:6.7}\normalfont
$\E[Z_1Z_1^*]=\E[Z_2Z_2^*]$, where
$Z_2=Z_0+\sum\limits_{i_1}Z^{(i_1)}_f+\Psi\big(\sum\limits_{i_2}Z^{(i_2)}_b+\sum\limits_{i_3}Z^{(i_3)}_{bf}\big)$.
\end{lemma}
\begin{proof}
	Appy Lemma \ref{lemma:7.5} and Lemma \ref{lemma:7.6}, we have
$$\begin{array}{rcl}
&& \E[Z_1Z_1^*]\\

&=&\E\Big[ (Z_0+\sum\limits_{i_1}Z^{(i_1)}_f+\sum\limits_{i_2}Z^{(i_2)}_b+\sum\limits_{i_3}Z^{(i_3)}_{bf})
(Z^*_0+\sum\limits_{i_1}Z^{(i_1)}_f+\sum\limits_{i_2}Z^{(i_2)}_b+\sum\limits_{i_3}Z^{(i_3)}_{bf})^*\Big]\\

&=&\E\Big[ (Z_0+\sum\limits_{i_1}Z^{(i_1)}_f+\sum\limits_{i_2}Z^{(i_2)}_b+\sum\limits_{i_3}Z^{(i_3)}_{bf})
(Z^*_0+\sum\limits_{i_1}Z^{(i_1)*}_f+\sum\limits_{i_2}Z^{(i_2)*}_b+\sum\limits_{i_3}Z^{(i_3)*}_{bf})\Big]\\

&=&\E\Big[\big(Z_0+\sum\limits_{i_1}Z^{(i_1)}_f+\Psi(\sum\limits_{i_2}Z^{(i_2)}_b+
\sum\limits_{i_3}Z^{(i_3)}_{bf})\big)
\big(Z^*_0+\sum\limits_{i_1}Z^{(i_1)*}_f+(\sum\limits_{i_2}Z^{(i_2)*}_b+
\sum\limits_{i_3}Z^{(i_3)*}_{bf})\big)\Big]\\

&=&\E\Big[\big(Z_0+\sum\limits_{i_1}Z^{(i_1)}_f+\Psi(\sum\limits_{i_2}Z^{(i_2)}_b+
\sum\limits_{i_3}Z^{(i_3)}_{bf})\big)
\big(Z^*_0+\sum\limits_{i_1}Z^{(i_1)*}_f+\Psi^*(\sum\limits_{i_2}Z^{(i_2)*}_b+
\sum\limits_{i_3}Z^{(i_3)*}_{bf})\big)\Big]\\

&=&\E\Big[\big(Z_0+\sum\limits_{i_1}Z^{(i_1)}_f+
\Psi(\sum\limits_{i_2}Z^{(i_2)}_b+\sum\limits_{i_3}Z^{(i_3)}_{bf})\big)
\big(Z^*_0+\sum\limits_{i_1}Z^{(i_1)*}_f+
(\Psi^*(\sum\limits_{i_2}Z^{(i_2)}_b+\sum\limits_{i_3}Z^{(i_3)}_{bf}))^*\big)\Big]\\

&=&\E[Z_2Z_2^*]\\

\end{array}
$$
\end{proof}
Notice that $\Psi(Z^{(i_2)}_b)$ is a simple Boolean product and $\Psi(Z^{(i_3)}_{bf})$ can be written as $B_{i_3}Z^{(i_3)}_{f'}$ where $Z^{(i_3)}_{f'}$ is the type $(3)$ of product in Proposition-Definition \ref{defprop}.
According to the length of the word appearing in the expression of
$Z_2$ defined in Lemma \ref{lemma:6.7}, we can then rewrite $Z_2$ as
$$Z_2=b1_\A+\sum\limits_{k=1}^{n}\sum\limits_{r=1}^{m_k}a_{k,r,1}a_{k,r,2}\cdots a_{k,r,k},$$
where 
$a_{k,r,p}\in D_{\omega_{k,r}(p)}$ (the right face) for $p\geq 2$,
and $a_{k,r,1}\in S_{\omega_{k,r}(p)}$ (either left face or right face),
$\E(a_{k,r,p})=0$ and 
$ \omega_{k,r}:\{1,\cdots, k\}\rightarrow \I$ such that 
$\omega_{k,r}(1)\neq \cdots\neq \omega_{k,r}(k)$. 

We now have the following result.
\begin{lemma} \label{lemma:7.8}\normalfont
	Let $a_1\in\A_{\omega_1(1)}$, $a_i\in \A_{\omega_1(i),\f}$ for $i=2,...,n$ such that 
	$\E(a_i)=\E(\widetilde{a}_i)=0$ and 
	$\omega_1:\{1,\cdots,n\} \rightarrow \I$, $\omega_1(1)\neq\cdots\neq\omega_1(n)$. 
	Let $\widetilde{a}_1\in\A_{\omega_2(1)}$, $\widetilde{a}_j\in \A_{\omega_2(j),\f}$ for $i=2,...,m$ such that $\omega_2:\{1,\cdots, m\} \rightarrow \I$, $\omega_2(1)\neq\cdots\neq\omega_1(m)$.
	We then have
	$$\E[a_1a_2\cdots a_{n-1}a_n\widetilde{a}_m\widetilde{a}_{m-1}\cdots\widetilde{a}_1]=\delta_{n,m}\E[a_1a_2\cdots \E[a_{n-1}\E[a_n\widetilde{a}_m]\widetilde{a}_{m-1}]\cdots\widetilde{a}_1].$$
\end{lemma}
\begin{proof}
When $n=0$ or $m=0$, there is nothing to prove. 
	It is sufficient to prove that $$\E[a_1a_2\cdots a_{n-1}a_n\widetilde{a}_m\widetilde{a}_{m-1}\cdots\widetilde{a}_1]=\delta_{n,m}\E[a_1a_2\cdots a_{n-1}\E[a_n\widetilde{a}_m]\widetilde{a}_{m-1}\cdots\widetilde{a}_1].$$
	
	Notice that if $\omega_1(n)\neq\omega_2(m)$, then by Proposition \ref{Free-boolean moments}
	and the definition of freeness, then, by Proposition 6.8, we have
	 $$\E[a_1a_2\cdots a_{n-1}a_n\widetilde{a}_m\widetilde{a}_{m-1}\cdots\widetilde{a_1}]=\E[a_n\widetilde{a}_m]=0.$$ 
	On the other hand, if $\omega_1(n)=\omega_2(m)$,  notice that $a_n\widetilde{a}_m-\E[a_n\widetilde{a}_m]\in \A_{\omega_1(n)}$, $\omega_1(n)\neq \omega_1(n-1)$ and $\omega_1(n)\neq\omega_2(m-1)$, then we have 
	$$\E[a_1a_2\cdots a_{n-1}(a_n\widetilde{a}_m-\E[a_n\widetilde{a}_m])\widetilde{a}_{m-1}\cdots\widetilde{a}_1]=0,$$
	which is the desired equation.
\end{proof}

Therefore, we have the following equation.

$$\E[Z_2Z_2^*]=bb^*+\sum\limits_{k=1}^{n}\E\left[\left(\sum\limits_{r=1}^{m_k}a_{k,r,1}a_{k,r,2}\cdots a_{k,r',k}\right)\left(\sum\limits_{r'=1}^{m_k}a_{k,r',1}a_{k,r',2}\cdots a_{k,r',k}\right)^*\right].$$
Moreover, Lemma \ref{lemma:7.8} implies that the term $\E\big[(a_{k,r,1}a_{k,r,2}\cdots a_{k,r',k})(a_{k,r',1}a_{k,r',2}\cdots a_{k,r',k})^*]$ is not vanishing only if $\omega_{k,r}=\omega_{k,r'}.$  Let $\sim$ be the equivalence relation on $\{1,\cdots,m_k\}$ such that  $l_1\sim l_2 $ if and only if $\omega_{k,l_1,k}=\omega_{k,l_2,k}$. Let $\{V_1,\cdots, V_s\}$ be the family of equivalence classes of $\{1,\cdots,m_k\}$. Then  
\[
\begin{array}{crl}
&& \E\left[\left(\sum\limits_{r=1}^{m_k}a_{k,r,1}a_{k,r,2}\cdots a_{k,r',k}\right)\left(\sum\limits_{r'=1}^{m_k}a_{k,r',1}a_{k,r',2}\cdots a_{k,r',k}\right)^*\right]   \\
&=&\E\left[\left(\sum\limits_{r\in V_l}a_{k,r,1}a_{k,r,2}\cdots a_{k,r',k}\right)\left(\sum\limits_{r'\in V_l}a_{k,r',1}a_{k,r',2}\cdots a_{k,r',k}\right)^*\right]
\end{array}
\]

To show that $E[Z_2Z_2^*]$ is positive, we just need to show that 
\[
\E\left[\left(\sum\limits_{r\in V_l}a_{k,r,1}a_{k,r,2}\cdots a_{k,r',k}\right)\left(\sum\limits_{r'\in V_l}a_{k,r',1}a_{k,r',2}\cdots a_{k,r',k}\right)^*\right]\geq 0
\]
for all $k$. 

This is exactly the circumstance in the proof of \cite[Proposition 3.5.6]{Sp1}.   
Therefore,  we have reduced the positivity question of  $\E[ZZ^*] $ 
to a known result in free probability context. Since $Z$ is arbitrary, we thus have the following theorem.

\begin{theorem}\normalfont
	Let $\{(C_i,D_i)\}_{i\in \I}$ be a family of  $B$-faces  in a $B$-probability space $(\A, \E)$ , which generates $\A$. We assume that $C_i, D_i$ are $*$-subalgebras of $\A$ for all $\A$. For each $i$, let $ \A_i$ be the $*$-algebras generated by $C_i, D_i$.  We assume that the restriction of $\E$ to $\A_i$ is positive.  If the family $\{(C_i,D_i)\}_{i\in \I}$ is free-Boolean independent with amalgamation over $B$, then $\E$ is positive.
\end{theorem}

\vspace{1cm}
{\bf Acknowledgement} The second-named author would like to thank Professor Hari Bercovici for the invitation to visit Indiana University,
where the authors can meet and part of this work was done. He also wants to thank Professor Alexandru Nica for his continued support. This project was partially supported by NSFC No. 11501423, 11431011.

\vspace{1cm}

\bibliographystyle{plain}

\bibliography{references}

\begin{thebibliography}{10}

\bibitem{BLS}
Marek Bo\.zejko, Michael Leinert, and Roland Speicher.
\newblock Convolution and limit theorems for conditionally free random
  variables.
\newblock {\em Pacific J. Math.}, 175(2):357--388, 1996.

\bibitem{CNS1}
Ian Charlesworth, Brent Nelson, and Paul Skoufranis.
\newblock Combinatorics of bi-freeness with amalgamation.
\newblock {\em Comm. Math. Phys.}, 338(2):801--847, 2015.

\bibitem{GHS}
Yinzheng Gu, Takahiro Hasebe, and Paul Skoufanis.
\newblock Bi-monotonic independence for pairs of algebras.
\newblock {\em arXiv:1708.05334}, 2017.

\bibitem{GS}
Yinzheng Gu and Paul Skoufanis.
\newblock Bi-boolean independence for pairs of algebras.
\newblock {\em arXiv:1703.03072, to appear in Complex Anal. Oper. Theory},
  2017.

\bibitem{GS2}
Yinzheng Gu and Paul Skoufranis.
\newblock Conditional bi-free independence with amalgamation,.
\newblock {\em arXiv:1609.07820, to appear in Int. Math. Res. Not. IMRN}, 2017.

\bibitem{GS1}
Yinzheng Gu and Paul Skoufranis.
\newblock Conditionally bi-free independence for pairs of faces.
\newblock {\em J. Funct. Anal.}, 273(5):1663--1733, 2017.

\bibitem{Liu}
Weihua Liu.
\newblock A noncommutative de {F}inetti theorem for boolean independence.
\newblock {\em J. Funct. Anal.}, 269(7):1950--1994, 2015.

\bibitem{Liu3}
Weihua Liu.
\newblock Free-boolean independence for pairs of algebras.
\newblock {\em arXiv:1710.01374}, 2017.

\bibitem{Ml}
Wojciech M\l~otkowski.
\newblock Operator-valued version of conditionally free product.
\newblock {\em Studia Math.}, 153(1):13--30, 2002.

\bibitem{Mu}
Naofumi Muraki.
\newblock The five independences as natural products.
\newblock {\em Infin. Dimens. Anal. Quantum Probab. Relat. Top.},
  6(3):337--371, 2003.

\bibitem{NS}
Alexandru Nica and Roland Speicher.
\newblock {\em Lectures on the combinatorics of free probability}, volume 335
  of {\em London Mathematical Society Lecture Note Series}.
\newblock Cambridge University Press, Cambridge, 2006.

\bibitem{Po}
Mihai Popa.
\newblock A combinatorial approach to monotonic independence over a
  {$C^*$}-algebra.
\newblock {\em Pacific J. Math.}, 237(2):299--325, 2008.

\bibitem{Po1}
Mihai Popa.
\newblock A new proof for the multiplicative property of the {B}oolean
  cumulants with applications to the operator-valued case.
\newblock {\em Colloq. Math.}, 117(1):81--93, 2009.

\bibitem{Rota}
Gian-Carlo Rota.
\newblock On the foundations of combinatorial theory. {I}. {T}heory of
  {M}\"obius functions.
\newblock {\em Z. Wahrscheinlichkeitstheorie und Verw. Gebiete}, 2:340--368
  (1964), 1964.

\bibitem{Sp}
Roland Speicher.
\newblock On universal products.
\newblock In {\em Free probability theory ({W}aterloo, {ON}, 1995)}, volume~12
  of {\em Fields Inst. Commun.}, pages 257--266. Amer. Math. Soc., Providence,
  RI, 1997.

\bibitem{Sp1}
Roland Speicher.
\newblock Combinatorial theory of the free product with amalgamation and
  operator-valued free probability theory.
\newblock {\em Mem. Amer. Math. Soc.}, 132(627):x+88, 1998.

\bibitem{SW}
Roland Speicher and Reza Woroudi.
\newblock Boolean convolution.
\newblock In {\em Free probability theory ({W}aterloo, {ON}, 1995)}, volume~12
  of {\em Fields Inst. Commun.}, pages 267--279. Amer. Math. Soc., Providence,
  RI, 1997.

\bibitem{Voi3}
Dan Voiculescu.
\newblock Operations on certain non-commutative operator-valued random
  variables.
\newblock {\em Ast\'erisque}, (232):243--275, 1995.
\newblock Recent advances in operator algebras (Orl\'eans, 1992).

\bibitem{Voi1}
Dan-Virgil Voiculescu.
\newblock Free probability for pairs of faces {I}.
\newblock {\em Comm. Math. Phys.}, 332(3):955--980, 2014.

\end{thebibliography}
%
%
%
%

\end{document}